\documentclass[11pt,reqno]{amsart}

\usepackage{amssymb,latexsym,verbatim,pdfsync,amsfonts,mathrsfs,enumerate}
\usepackage{graphicx}
\usepackage{color}

\newcommand{\bbC}{{\mathbb C}}
\newcommand{\bbN}{{\mathbb N}}
\newcommand{\bbR}{{\mathbb R}}
\newcommand{\bbZ}{{\mathbb Z}}

\def\cG{{\mathcal G}}
\def\cH{{\mathcal H}}
\def\cI{{\mathcal I}}

\def\cM{{\mathcal M}}
\def\cN{{\mathcal N}}

\def\cP{{\mathcal P}}
\def\cQ{{\mathcal Q}}

\def\cW{{\mathcal W}}
\def\index{r}

\def\Re{\operatorname{Re}}
\def\Im{\operatorname{Im}}

\newtheorem{theorem}{Theorem}[section]
\newtheorem{proposition}[theorem]{Proposition}
\newtheorem{corollary}[theorem]{Corollary}
\newtheorem{lemma}[theorem]{Lemma}
\newtheorem{definition}[theorem]{Definition}
\newtheorem{remark}[theorem]{Remark}

\begin{document}

\title{On a Higher-Dimensional Worm Domain and its Geometric Properties}

\author{Steven G. Krantz, Marco M. Peloso, and Caterina Stoppato}

\begin{abstract}
We construct new $3$-dimensional variants of the classical Diederich-Forn{\ae}ss worm domain. We show that they are smoothly bounded, pseudoconvex, and have nontrivial Nebenh\"{u}lle. We also show that their Bergman projections do not preserve the Sobolev space for sufficiently large Sobolev indices.
\end{abstract}

\maketitle

\begin{center}{\footnotesize
}
\end{center}


\section{Introduction}

In the paper \cite{DFo1}, Diederich and Forn{\ae}ss constructed a
class of smoothly bounded, pseudoconvex domains in $\bbC^2$, denoted
here as $\Omega_\mu$ for $\mu>0$, having nontrivial Nebenh\"{u}lle,
i.e., such that the intersection of all pseudoconvex domains
containing the closure $\overline{\Omega}_\mu$ properly contains
$\Omega_\mu$. Such a domain $\Omega_\mu$ in the two complex variables
$z_1,z_2$, now called (smooth) {\em worm domain}, is a union of disks
in the $z_1$-plane, each centered at the unitary point
$e^{i\log|z_2|^2}$. In other words, $\Omega_\mu$ is a union of disks
winding about the origin in the $z_1$-plane according to a rotation
angle $\log|z_2|^2$. The radii of these disks are $1$ when
$\log|z_2|^2$ belongs to the interval $I_\mu=(-\mu,\mu)$, and fade to
$0$ as $|\log|z_2|^2|$ ranges from $\mu$ to some finite
$\mu'>\mu$. This ground-breaking idea of Diederich and Forn{\ae}ss
makes every defining function for $\Omega_\mu$ not globally
plurisubharmonic. Moreover, the boundary $\partial\Omega_\mu$ is
strongly Levi-pseudoconvex at every point, apart from the
$1$-dimensional exceptional set $\{0\}\times A$, where $A$ is an
annulus centered at the origin in the $z_2$-plane, determined by the
condition $\log|z_2|^2\in I_\mu$. Thanks to these features, the domain
$\Omega_\mu$ also serves as a counterexample to a number of important
geometric phenomena. We mention in particular the
article~\cite{Barrett-Acta} where, for $\mu>\frac\pi2$, Barrett showed
that the Bergman projection on $\Omega_\mu$ fails to preserve the
Sobolev space $W^{\index,2}(\Omega_\mu)$ when
$\index\geq\nu=\frac{\pi}{2\mu}$.  Based on Barrett's result, with a
clever and far-reaching argument, Christ~\cite{Chr} proved that the
Bergman projection and the Neumann operator on $\Omega_\mu$ fail to be
globally regular. These properties are well described also in the monographs~\cite{ChSh,Straube-lectures}. Besides the non-trivial Nebenh\"ulle and the already mentioned questions of global regularity, the worm domains  are ``testing ground'' for several other problems. Here we just mention some of the most recent papers on these topics, \cite{Cu-Sa,DallAra-Mongodi,KPS2016,KPS2019,KMPS2023,Liu2019} and references therein.
 
It is a matter of great interest to fully understand the worm, and to put it in a more general context. For instance, the recent article~\cite{ADF2023} (referring to~\cite{Barrett-sectorial}) highlights a crucial geometric feature of the worm domain and uses it to generalize the construction of Diederich and Forn{\ae}ss to a wider class of $2$-dimensional domains. It is also interesting to investigate analogous constructions in higher dimensions. The article~\cite{BaSa} constructed analogs of the worm domain within $\bbC^n$. There, the boundary is strongly Levi-pseudoconvex at every point apart from a $1$-dimensional exceptional set $\{{\bf0}\}\times A$, where ${\bf0}$ is the origin in $\bbC^{n-1}$ and $A$ is an annulus centered at the origin in the $z_n$-plane.

In the present work, we construct for every $\mu>0$ a new class
$\mathscr{C}_\mu$ of smoothly bounded, pseudoconvex domains in
$\bbC^3$. Every $\Omega\in\mathscr{C}_\mu$ is, again, a union of disks
winding about the origin in the $z_1$-plane; but now the rotation
angle is $\log|z_2z_3|^2$, where $(z_2,z_3)$ is confined to a bounded
neighborhood of the product of annuli $A\times A$ determined by the
condition $(\log|z_2|^2,\log|z_3|^2)\in I_\mu\times I_\mu$. The
boundary $\partial\Omega$ turns out to be weakly Levi-pseudoconvex at
every point $z=(z_1,z_2,z_3)\in\partial\Omega$ with $(z_2,z_3)\in
A\times A$. Moreover, $\partial\Omega$ includes the $2$-dimensional
complex manifold $\{0\}\times A\times A$. We prove that $\Omega$, just
like the original worm domain of~\cite{DFo1}, has nontrivial
Nebenh\"ulle when $\mu>\frac\pi2$. We also prove, following the lines
of~\cite{Barrett-Acta}, that the Bergman projection of $\Omega$ does
not preserve the Sobolev space $W^{\index,2}(\Omega)$ when
$\mu>\frac\pi2,\index\geq\nu=\frac{\pi}{2\mu}$.

Our construction requires some effort, which is unsurprising if we take into account Boas and Straube's result~\cite[Corollary 5.16]{Straube-lectures}. Let us recall the substance of this result, originally proven in~\cite{BoSt}. Let $\Omega\subseteq\bbC^n$ be a smoothly bounded domain and assume there exists a smooth real submanifold $M$ of the boundary $\partial\Omega$ such that: $M$ includes all infinite-type points of $\partial\Omega$; the real tangent space of $M$ at each $p\in M$ is included in the null space of the Levi form of $\partial\Omega$ at $p$. A trivial de Rham cohomology for $M$ is a sufficient, but not necessary, condition for the regularity of the Bergman projection and Neumann operator of $\Omega$. Regularity is still guaranteed when $M$ has nontrivial cohomology, provided a specific cohomology class vanishes. Thus, a possible irregularity of the Bergman projection is not simply due to the presence of the manifold $M$ but rather to \emph{how} $M$ is embedded in $\partial\Omega$.

The paper is structured as follows:

Section~\ref{sec:construction} constructs a wide family of domains $\cW_\eta$ in $\bbC^3$, varying with the choice of a function $\eta$. Two conditions on $\eta$ are determined, to make the boundary $\partial\cW_\eta$ smooth except at (possible) boundary points $z$ with $z_2z_3=0$. A third condition on $\eta$ is then added, to make $\cW_\eta$ smoothly bounded.

Section~\ref{sec:pseudoconvex} starts with a characterization of pseudoconvexity of $\cW_\eta$ by means of two inequalities in the first and second derivatives of $\eta$. Some unbounded or non-smooth examples are immediately provided. Then $\mathscr{C}_\mu$ is defined as the class of the $\cW_\eta$ with $\eta$ fulfilling all three conditions for boundedness and smoothness, as well as the inequalities for pseudoconvexity. Finally, $\mathscr{C}_\mu$ is proven not empty by constructing an explicit class of examples in Theorem~\ref{thm:main}, which is the main result of this work.

Section~\ref{sec:study} is devoted to the study of any smoothly bounded, pseudoconvex domain $\Omega$ in the class $\mathscr{C}_\mu$. In particular, for the case when $\mu>\frac\pi2$: Theorem~\ref{non-density} shows that $\Omega$ has nontrivial Nebenh\"ulle; Theorem~\ref{non-regularity} concerns the irregularity of the Bergman projection of $\Omega$.

Section~\ref{sec:proof} comprises the proof of Theorem~\ref{non-regularity}, as well as several tools used in the proof.\\\\


\section{Basic construction}\label{sec:construction}

We define a family of domains $\cW_\eta$ in dimension $3$, varying with the choice of a function $\eta$. Making different assumptions on $\eta$, we will later prove different properties of $\cW_\eta$. As customary, the word domain here means a nonempty connected open subset of $\bbC^3$.

\begin{definition}\label{def:basic}{\em 
Fix a real number $\mu > \pi$. Let $\eta: \bbR^2 \to [0, +\infty)$ be upper semicontinuous, with $\eta^{-1}(0)$ including the square $(-\mu,\mu)\times(-\mu,\mu)$ and with $\eta^{-1}([0,1))$ path-connected. Set $\bbC^*:=\bbC\setminus\{0\}$ and
\[\cW = \cW_\eta= \left \{ (z_1, z_2, z_3) \in \bbC \times \bbC^* \times \bbC^*: \left | z _1 - e^{i \log |z_2 z_3|^2} \right|^2 < 1 - \eta(\log|z_2|^2,\log|z_3|^2) \right \} \, .\]
Equivalently, $\cW$ is the subset of $\bbC \times \bbC^* \times \bbC^*$ having the upper semicontinuous function
\[\rho(z_1,z_2,z_3) = |z_1|^2 - 2 \Re (z_1 e^{-i \log|z_2 z_3|^2} ) + \eta(\log|z_2|^2,\log|z_3|^2)\]
as a defining function.
}
\end{definition}

In what follows, we let $\bigtriangleup(P,r)$ denote the disk in the complex plane with
center $P$ and radius $r$.  We also let $\bigtriangleup^*(P,r)$ denote the punctured disk.
Finally, we let $A(P, r_1, r_2)$ denote the planar annulus with center $P$ and radii
$r_1 < r_2$.

The next remark decomposes $\cW$ into a family of disks that rotate about the origin in $\bbC$.

\begin{remark} {\em 
If we set
\[R(z_2, z_3) := 1 - \eta(\log|z_2|^2,\log|z_3|^2) \, ,\]
then
\[\cW = \bigcup_{R(z_2, z_3) > 0} \bigtriangleup  \bigl ( e^{i \log|z_2 z_3|^2}, R(z_2, z_3) \bigr ) \times \{(z_2, z_3)\} \, .\]
As a result, if we let $\pi_j$ denote projection on the $j$th variable, we have $\pi_1(\cW) =  \bigtriangleup^*(0,2)$.
}
\end{remark}

\begin{remark} {\em
Let us show that $\cW$ is a domain in $\bbC^3$.

The set $\cW$ is an open subset of $\bbC \times \bbC^* \times \bbC^*$ (whence an open subset of $\bbC^3$) because its defining function $\rho:\bbC \times \bbC^* \times \bbC^*\to\bbR$ is upper semicontinuous.

Additionally, $\cW$ is path-connected because for every $z = (z_1,z_2,z_3) \in \cW$ we can find a path in $\cW$ joining $z$ to the point $(1,1,1)$. We begin by observing that, since $\eta^{-1}([0,1))$ is path-connected by hypothesis, there exists a path $\gamma=(\gamma_2,\gamma_3)$ from $(\log|z_2|^2,\log|z_3|^2)$ to $(0,0)$ in $\eta^{-1}([0,1))$. Now, $z = (z_1,z_2,z_3)$ can be joined by a line segment within $\cW$ to the point $(e^{i \log|z_2 z_3|^2},z_2,z_3)$, which can be joined to the point $(1,\frac{z_2}{|z_2|},\frac{z_3}{|z_3|})$ by means of the path
\[[0,1]\ni s\mapsto\left(e^{i\gamma_2(s)+i\gamma_3(s)},e^{\gamma_2(s)/2}\frac{z_2}{|z_2|},e^{\gamma_3(s)/2}\frac{z_3}{|z_3|}\right)\in\cW\,.\]
Finally, every point of the form $(1,e^{it_2},e^{it_3})$ can be joined to $(1,1,1)$ by means of the path
\[[0,1]\ni
  s\mapsto\left(1,e^{i(1-s)t_2},e^{i(1-s)t_3}\right)\in\cW\,.\]
}
\end{remark}

The main body of $\cW$ consists of those disks that have radius $1$. Because of the choices made in Definition~\ref{def:basic}, the set $R^{-1}(1)$ of those $(z_2,z_3) \in (\bbC^*)^2$ such that $\eta(\log|z_2|^2,\log|z_3|^2)=0$ includes
\[A(0, e^{-\mu/2}, e^{\mu/2}) \times A(0, e^{-\mu/2}, e^{\mu/2}).\]
Restricting to this main body for the sake of simplicity, we can study the fiber $\pi_1^{-1}(z_1)$ over each $z_1\in\bigtriangleup^*(0,2)$ as follows.

\begin{proposition} \sl
Fix any point $z_1 = \rho_0 e^{i\theta_0}$ with $0 < \rho_0 < 2$ and $\theta_0 \in \bbR$.  For any $k \in \bbZ$, let
\[r_k^{\pm} := \exp(\theta_0/2 \pm [\arccos(\rho_0/2)]/2 + \pi k) \, .\]
If $(z_2, z_3) \in R^{-1}(1)$, then $(z_1,z_2,z_3)$ belongs to $\pi_1^{-1}(z_1)$ if and only if 
\[z_2 z_3 \in \bigcup_{k \in \bbZ} A(0, r_k^-, r_k^+) \, .\]
\end{proposition}

\begin{proof}
If $(z_2, z_3) \in R^{-1}(1)$, then $\eta(\log|z_2|^2,\log|z_3|^2)=0$.  Also
\begin{eqnarray*}
\rho(z_1,z_2,z_3) & = & |z_1|^2 - 2\Re (z_1 e^{-i\log|z_2 z_3|^2})  \\
                  & = & \rho_0^2 - 2\rho_0 \Re (e^{i(\theta_0 - \log|z_2z_3|^2)} \\
		  & = & \rho_0 (\rho_0 -2 \cos (\theta_0 - \log|z_2z_3|^2)) \, .
\end{eqnarray*}
This quantity is negative if and only if $\cos(\theta_0 - \log|z_2 z_3|^2) > \rho_0/2$.  This is, in turn,
equivalent to having
\[\theta_0 - \arccos(\rho_0/2) + 2\pi k < \log |z_2z_3|^2 < \theta_0 + \arccos(\rho_0/2) + 2\pi k \, ,\]
that is, $r_k^- < |z_2 z_3| < r_k^+$ for some $k \in \bbZ$.
\end{proof}

Let us set a few notations, which will prove useful later.

\begin{definition}\label{def:etacirc}{\em
For each $z=(z_1,z_2,z_3)\in \bbC \times \bbC^* \times \bbC^*$, we set $L:=\log|z_2z_3|^2$. Moreover, referring to the map $(t_2,t_3)\mapsto\eta(t_2,t_3)$, we set 
\begin{align*}
\eta^\circ(z)&:=\eta(\log|z_2|^2,\log|z_3|^2)
\end{align*}
and, if $\eta$ is smooth,
\begin{align*}
\eta'_{j}(z)&:=\frac{\partial\eta}{\partial t_j}(\log|z_2|^2,\log|z_3|^2)\\
\eta''_{j,k}(z)&:=\frac{\partial^2\eta}{\partial t_j\partial t_k}(\log|z_2|^2,\log|z_3|^2)
\end{align*}
for all $j,k\in\{2,3\}$.}
\end{definition}

\begin{remark} \label{rmk:firstderivatives} {\em
Fix $(z_1, z_2, z_3)\in \bbC \times \bbC^* \times \bbC^*$. By direct computation, we see that
\[\frac{\partial}{\partial z_1} \rho(z_1, z_2, z_3) = \overline{z}_1 - e^{-iL}\,.\]
Moreover, if $\eta$ is smooth, then
\[ \frac{\partial}{\partial z_j} \rho(z_1, z_2, z_3) = -2 \Im (z_1 e^{-iL}) z_j^{-1} + \eta'_jz_j^{-1} \ \ \hbox{for} \ \ j \in \{2,3\}\,.\]
}
\end{remark}

Now, for appropriate choices of $\eta$, we can prove a first significant result about the geometry of $\cW_\eta$.

\begin{proposition}\label{prop:complextangent}   \sl
Under the following additional assumptions:
\begin{itemize}
\item[(I)] $\eta$ is smooth in an open neighborhood of $\eta^{-1}([0,1])$;
\item[(II)] at each point of $\eta^{-1}(1)$ (if any), the gradient of $\eta$ does not vanish;
\end{itemize}
then $\cW_\eta$ is a domain in $\bbC^3$ whose boundary is smooth except at the boundary points $z$ with $z_2z_3=0$ (if any). For each $z=(z_1, z_2, z_3) \in \partial \cW_\eta$, if $z_1 \neq e^{iL}$ and $z_2z_3\neq0$, then the complex tangent space to $\partial \cW_\eta$ at $z$ is spanned by the vectors
\begin{align*}
&\hbox{\bf v}(z_1, z_2, z_3) := \left ( \begin{array}{c}
                            2\Im (z_1 e^{-iL})-\eta'_2 \\
			    z_2(\overline{z}_1 - e^{-i L}) \\
			          0  \\
				\end{array}
			\right )\,,\qquad
&\hbox{\bf w}(z_1, z_2, z_3) := \left ( \begin{array}{c}
			    2\Im (z_1 e^{-iL}) - \eta'_3\\
	    			  0  \\
			    z_3(\overline{z}_1 - e^{-i L}) \\
				\end{array}
			\right ) \, .
\end{align*}
In the (non generic) case when $z_1 = e^{iL}$ and $z_2z_3\neq0$, the complex tangent space to $\partial \cW_\eta$ at $z$ is spanned by the vectors
\begin{align*}
&\hbox{\bf v}(z_1, z_2, z_3) := \left ( \begin{array}{c}
                            1 \\
			   0 \\
			   0  \\
				\end{array}
			\right )\,,\qquad
&\hbox{\bf w}(z_1, z_2, z_3) := \left ( \begin{array}{c}
			    0\\
	    		z_2\eta'_3\\
			-z_3\eta'_2
			\end{array}
			\right ) \, .
\end{align*}
Under the further additional assumption:
\begin{itemize}
\item[(III)] $\eta^{-1}([0,1])$ is a bounded subset of $\bbR^2$;
\end{itemize}
then $\cW_\eta$ is a smoothly bounded domain in $\bbC^3$.
\end{proposition}

\begin{proof}
We first work under assumptions (I) and (II) only. Suppose that $z=(z_1, z_2, z_3)$ is a boundary point of $\cW_\eta$ with $z_2z_3\neq0$, so that
 $\rho(z) = 0$.  If $\partial \rho/\partial z_1(z) = 0$, then $z_1 = e^{iL}$.
 This in turn implies that 
\[0 =  \rho(z) = -1 + \eta(\log|z_2|^2,\log|z_3|^2) \, .\]
 As a result, $(\log|z_2|^2,\log|z_3|^2) \in \eta^{-1}(1)$. Thus, $\eta'_j \neq 0$ for some $j \in \{2,3\}$. The equality $z_1 = e^{i L}$ also implies that 
\[\frac{\partial}{\partial z_j} \rho(z) = \eta'_jz_j^{-1} \, .\]   
In conclusion, $\partial \rho/\partial z_j (z) \ne 0$. Since the gradient of $\rho$ does not vanish at $z$, the boundary $\partial \cW_\eta$ is smooth at $z$. The generators of the complex tangent space can be computed directly.

We now take all assumptions (I),(II),(III). In particular, there exists $\mu'>\mu>0$ such that $\eta^{-1}([0,1])$ is included in a closed disk of radius $\mu'$ centered at the origin in $\bbR^2$. Now, if $z=(z_1, z_2, z_3)\in\cW_\eta$, then $R(z_2,z_3)>0$, whence $\eta(\log|z_2|^2,\log|z_3|^2)<1$ and $e^{-\mu'/2}<|z_2|,|z_3|<e^{\mu'/2}$. Thus,
\[\cW_\eta \subseteq \triangle^*(0,2)\times A(0,e^{-\mu'/2},e^{\mu'/2})\times A(0,e^{-\mu'/2},e^{\mu'/2})\]
is a bounded domain and there exists no boundary point $w=(w_1, w_2, w_3)$ of $\cW_\eta$ having $w_2w_3=0$.
\end{proof}

\begin{remark} \label{rmk:complextangent}
Under our additional hypotheses (I) and (II), we have that the complex tangent space to $\partial \cW_\eta$ at each boundary point of the form $(0, w_2, w_3)$ equals 
\[\hbox{\rm span}_\bbC \left \{ \left ( \begin{array}{c}
                                          0 \\
					  1 \\
					  0
				     \end{array}
			     \right ) ,
			     \left ( \begin{array}{c}
                                          0 \\
					  0 \\
					  1
				     \end{array}
			     \right )  \right \} \, .\]\\
\end{remark}


\section{Study of pseudoconvexity}\label{sec:pseudoconvex}

As it happened with the original worm domain of \cite{DFo1}, we will now study the pseudoconvexity of $\cW_\eta$. We begin with a useful remark.

\begin{remark}\label{rmk:localdefiningfunction}
Fix $w=\big(w_1,w_2,w_3\big) \in \overline{\cW}$ with $\varepsilon:=\big|w_2\,w_3\big|>0$. Then
\[U_{w}:=\Big\{(z_1,z_2,z_3) \in \bbC^3\ :\ (z_2\,z_3)^2 \in \bigtriangleup\big((w_2\,w_3)^2,\varepsilon^2\big)\Big\}\]
is an open neighborhood of $w$ in $\bbC^3$. Moreover, $\bigtriangleup\big((w_2\,w_3)^2,\varepsilon^2\big)$ is a simply connected open subset of $\bbC$ missing the origin, where a branch of logarithm $\log$ (whence a branch $\arg$ of the argument function) is defined. In $U_{w}$, the product
\begin{align*}
\widetilde{\rho}(z_1,z_2,z_3) :=& e^{\arg(z_2\,z_3)^2}\rho(z_1,z_2,z_3)\\
=& e^{\arg(z_2\,z_3)^2}|z_1|^2 - 2 \Re (z_1 e^{-i \log(z_2\,z_3)^2} ) + e^{\arg(z_2\,z_3)^2}\eta(\log|z_2|^2,\log|z_3|^2)
\end{align*}
is a local defining function for $\cW_\eta$.
\end{remark}

We are now ready for our first main theorem. In the proof, and throughout the paper, we will use the notation
\begin{equation}\label{eq:hermitianform}
h_M({\bf u}_1,{\bf u}_2):={\bf u}_1^tM\bar{\bf u}_2
\end{equation}
for every complex square matrix $M$ of order $3$ and any ${\bf u}_1,{\bf u}_2\in\bbC^3$. We will also use the notations $\eta^\circ,\eta'_j,\eta''_{jk}:\bbC \times \bbC^* \times \bbC^*\to [0, +\infty)$ set in Definition~\ref{def:etacirc}.

\begin{theorem}\label{thm:pseudoconvex}
We take the extra assumptions (I),(II) on $\eta$ and fix $w\in\partial\cW_\eta$ with $w_2\,w_3\neq0$. If
\begin{align}
&\eta^\circ+\eta''_{22}\geq0\label{eq:pseudoconvex1}\\
&\eta''_{22}\eta''_{33}-(\eta''_{23})^2+\eta^\circ(\eta''_{22}+\eta''_{33}-2\eta''_{23})-(\eta'_2-\eta'_3)^2\geq0\label{eq:pseudoconvex3}
\end{align}
at $w$, then $\cW_\eta$ is Levi pseudoconvex at $w$. If, moreover, conditions~\eqref{eq:pseudoconvex1} and~\eqref{eq:pseudoconvex3} are fulfilled in a neighborhood of $w$, then the local defining function $\widetilde{\rho}$ constructed in Remark~\ref{rmk:localdefiningfunction} is plurisubharmonic near $w$.
\end{theorem}

\begin{proof}
Fix $w=\big(w_1,w_2,w_3\big) \in \overline{\cW_\eta}$ with $\big|w_2\,w_3\big|>0$. Let us adopt the notations of Remark~\ref{rmk:localdefiningfunction} and argue locally in $U_{w}$, studying the complex Hessian matrix of the local defining function $\widetilde{\rho}$ is plurisubharmonic. The second addend in $\widetilde{\rho}(z)$, namely $- 2 \Re (z_1 e^{-i \log(z_2\,z_3)^2} )$, is the real part of a holomorphic function, whence a real-valued pluriharmonic function. We are left with studying the complex Hessian matrix of the sum of $m(z):=e^{\arg(z_2\,z_3)^2}|z_1|^2$ and $n(z):=e^{\arg(z_2\,z_3)^2}\eta(\log|z_2|^2,\log|z_3|^2)$. Let us compute the complex Hessian matrix of $m(z)$ at $z=(z_1,z_2,z_3)\in U_{w}$. By direct computation we see that
\begin{eqnarray*}
 \frac{\partial}{\partial z_1} m(z) & = & e^{\arg(z_2\,z_3)^2}\overline{z}_1\\
 \frac{\partial}{\partial z_j} m(z) & = & e^{\arg(z_2\,z_3)^2}\frac{|z_1|^2}{iz_j}\ \ \hbox{for} \ \ j \in \{2,3\}\,,
 \end{eqnarray*}
whence the complex Hessian matrix of $m(z)$ is $e^{\arg(z_2\,z_3)^2}M(z)$, where
\[M(z):=\begin{pmatrix}
1&i\frac{\bar z_1}{\bar z_2}&i\frac{\bar z_1}{\bar z_3}\\
-i\frac{z_1}{z_2}&\frac{|z_1|^2}{|z_2|^2}&\frac{|z_1|^2}{z_2\bar z_3}\\
-i\frac{z_1}{z_3}&\frac{|z_1|^2}{\bar z_2 z_3}&\frac{|z_1|^2}{|z_3|^2}
\end{pmatrix}\,.\]
Let us consider the vectors
\[{\bf p}
=\begin{pmatrix}
1\\0\\0
\end{pmatrix}\,,\quad
{\bf n}_2
=\begin{pmatrix}
i z_1\\z_2\\0
\end{pmatrix}\,,\quad
{\bf n}_3
=\begin{pmatrix}
i z_1\\0\\z_3
\end{pmatrix}\,.
\]
Since
\begin{align*}
M\bar{\bf p}=\begin{pmatrix}
1\\
-i\frac{z_1}{z_2}\\
-i\frac{z_1}{z_3}
\end{pmatrix}\,,\quad
M\bar{\bf n}_2=\begin{pmatrix}
0\\
0\\
0
\end{pmatrix}\,,\quad
M\bar{\bf n}_3=\begin{pmatrix}
0\\
0\\
0
\end{pmatrix}\,,
\end{align*}
the vectors ${\bf p},{\bf n}_2,{\bf n}_3$ form an orthogonal basis with respect to $h_M$, the real number $h_M({\bf p},{\bf p})=1$ is positive and $\bar{\bf n}_2,\bar{\bf n}_3$ belong to the null space of $M$. In particular, $M$ is positive semidefinite and $m$ is plurisubharmonic.

Now assume $\eta$ to fulfill the assumptions (I),(II) of Proposition~\ref{prop:complextangent}. The equalities
\begin{eqnarray*}
 \frac{\partial}{\partial z_1} n(z) & = & 0\\
 \frac{\partial}{\partial z_j} n(z) & = & e^{\arg(z_2\,z_3)^2}\left(\frac{\eta^\circ}{iz_j}+\frac{\eta'_j}{z_j}\right)\ \ \hbox{for} \ \ j \in \{2,3\}\\
 \frac{\partial^2}{\partial z_j\partial\bar z_k} n(z) & = & \frac{e^{\arg(z_2\,z_3)^2}}{-i\bar z_k}\left(\frac{\eta^\circ}{iz_j}+\frac{\eta'_j}{z_j}\right)
+ e^{\arg(z_2\,z_3)^2}\left(\frac{\eta'_k}{iz_j\bar z_k}+\frac{\eta''_{j,k}}{z_j\bar z_k}\right) \ \ \hbox{for} \ \ j,k \in \{2,3\}
 \end{eqnarray*}
imply that the complex Hessian matrix of $n(z)$ is $e^{\arg(z_2\,z_3)^2}N(z)$, where
\[N(z):=\begin{pmatrix}
0&0&0\\
0&\frac{\eta^\circ+\eta''_{22}}{|z_2|^2}&\frac{\eta^\circ+\eta''_{23}+i(\eta'_2-\eta'_3)}{z_2\bar z_3}\\
0&\frac{\eta^\circ+\eta''_{23}+i(\eta'_3-\eta'_2)}{\bar z_2 z_3}&\frac{\eta^\circ+\eta''_{33}}{|z_3|^2}
\end{pmatrix}\,.\]
Since
\begin{align*}
N\bar{\bf p}=\begin{pmatrix}
0\\
0\\
0
\end{pmatrix}\,,\quad
N\bar{\bf n}_2=\begin{pmatrix}
0\\
\frac{\eta^\circ+\eta''_{22}}{z_2}\\
\frac{\eta^\circ+\eta''_{23}+i(\eta'_3-\eta'_2)}{z_3}
\end{pmatrix}\,,\quad
N\bar{\bf n}_3=\begin{pmatrix}
0\\
\frac{\eta^\circ+\eta''_{23}+i(\eta'_2-\eta'_3)}{z_2}\\
\frac{\eta^\circ+\eta''_{33}}{z_3}
\end{pmatrix}\,,
\end{align*}
we conclude that, for all $\alpha_1,\alpha_2,\alpha_3\in\bbC$,
\[h_{M+N}(\alpha_1{\bf p}+\alpha_2{\bf n}_2+\alpha_3{\bf n}_3,\alpha_1{\bf p}+\alpha_2{\bf n}_2+\alpha_3{\bf n}_3)=|\alpha_1|^2+\begin{pmatrix}\alpha_2&\alpha_3\end{pmatrix}C\begin{pmatrix}\overline{\alpha}_2\\\overline{\alpha}_3\end{pmatrix}\,,\]
where
\[C=C(\log|z_2|^2,\log|z_3|^2):=\begin{pmatrix}
\eta^\circ+\eta''_{22}&\eta^\circ+\eta''_{23}+i(\eta'_2-\eta'_3)\\
\eta^\circ+\eta''_{23}+i(\eta'_3-\eta'_2)&\eta^\circ+\eta''_{33}
\end{pmatrix}\,.\]
Thus, the sum $m+n$ (whence $\widetilde{\rho}$) is plurisuharmonic near $w$ if, and only if, $C(\log|z_2|^2,\log|z_3|^2)$ is positive semidefinite for $z$ near $w$. This happens if, and only if, both the $(1,1)$ entry $\eta^\circ+\eta''_{22}$ of $C$ and the determinant
\begin{align*}
\det C &= (\eta^\circ+\eta''_{22})(\eta^\circ+\eta''_{33})-(\eta^\circ+\eta''_{23})^2-(\eta'_2-\eta'_3)^2\\
&=\eta''_{22}\eta''_{33}-(\eta''_{23})^2+\eta^\circ(\eta''_{22}+\eta''_{33}-2\eta''_{23})-(\eta'_2-\eta'_3)^2
\end{align*}
are non-negative near $w$. This is the same as asking for conditions~\eqref{eq:pseudoconvex1} and~\eqref{eq:pseudoconvex3} to be fulfilled near $w$. Finally, in case $w\in\partial\cW_\mu$: if conditions~\eqref{eq:pseudoconvex1} and~\eqref{eq:pseudoconvex3} are fulfilled \emph{at} $w$, then $h_{M+N}$ is non-negative on $\bbC^3\times\bbC^3$, whence $\cW_\mu$ is Levi pseudoconvex at $w$.
\end{proof}

Theorem~\ref{thm:pseudoconvex} allows us to provide our first three examples of pseudoconvex domains in $\bbC^3$.

\begin{corollary}\label{cor:unbounded}\phantom{Examples}
\begin{enumerate}[(i)]
\item If $\eta\equiv0$, then $\cW_\eta=\cW_0$ is an unbounded pseudoconvex domain, whose boundary is smooth except at the points $w$ with $w_2w_3=0$.
\item If $\eta$ is the characteristic function of the complement of the square $(-\mu,\mu)\times(-\mu,\mu)$, then $\cW_\eta$ is a bounded pseudoconvex domain, whose boundary is smooth except at the points $w$ with $\log|w_2|^2=\pm\mu$ or $\log|w_3|^2=\pm\mu$.
\item $\cW_\eta$ is an unbounded pseudoconvex domain with smooth boundary if we pick
\[\eta(t_2,t_3)=\phi(t_2+t_3),\mathrm{\quad i.e.,\quad}\eta^\circ=\phi(\log|z_2z_3|^2)\,,\]
for a smooth, strictly convex and even function $\phi$ having $\phi^{-1}(0)=[-\mu,\mu]$ and $\phi^{-1}([0,1])=[-\mu',\mu']$ for some real number $\mu'>\mu$.
\end{enumerate}
\end{corollary}

\begin{proof}
To prove the third statement, we will apply Theorem~\ref{thm:pseudoconvex}. The first statement will then follow and, in turn, imply the second statement.
\begin{itemize}
\item[(iii)] If $\eta(t_2,t_3)=\phi(t_2+t_3)$, then the additional assumptions (I) and (II) are fulfilled. Moreover, for all $z\in\bbC\times\bbC^*\times\bbC^*$ we find that $\eta'_2=\eta'_3=\phi'(\log|z_2z_3|^2)$ and that $\eta''_{22}=\eta''_{33}=\eta''_{23}=\phi''(\log|z_2z_3|^2)$. Thus, inequality~\eqref{eq:pseudoconvex1} is fulfilled and an equality holds in~\eqref{eq:pseudoconvex3}. Finally, each $w=\big(w_1,w_2,w_3\big)$ in the finite boundary $\partial\cW_\eta$ has $\log\left|w_2\,w_3\right|^2\in[-\mu',\mu']$, whence $w_2\,w_3\neq0$.
\item[(i)] Pick $\phi$ as in case (iii). For $n\in\bbN^*$, we set $\eta_n:=\frac{\phi}{n}$ and remark that $\{\cW_{\eta_n}\}_{n\in\bbN^*}$ is an increasing sequence of pseudoconvex domains with smooth boundaries. It follows that
\[\cW_0=\bigcup_{n\in\bbN^*}\cW_{\eta_n}\]
is pseudoconvex, too. Additionally, the boundary $\partial\cW_0$ coincides locally near each $w\in\partial\cW_0$ with $w_2\,w_3\neq0$ with the boundary $\partial\cW_{\eta_n}$ for some $n\in\bbN^*$.
\item[(ii)] If $\eta$ is the characteristic function of the complement of the square $(-\mu,\mu)\times(-\mu,\mu)$, then $\cW_\eta$ is the intersection between the unbounded pseudoconvex domain $\cW_0$ of $\bbC\times\bbC^*\times\bbC^*$ and the pseudoconvex domain $\bbC\times A(0, e^{-\mu/2}, e^{\mu/2})\times A(0, e^{-\mu/2}, e^{\mu/2})$ of $\bbC^3$. It follows that $\cW_\eta$ is a pseudoconvex domain in $\bbC^3$. Every boundary point $w=\big(w_1,w_2,w_3\big)$ either is a boundary point of $\cW_0$ with $w_2\,w_3\neq0$ or has $|w_j|=e^{\pm\mu/2}$ for some $j\in\{2,3\}$. The thesis follows.\qedhere
\end{itemize}
\end{proof}

For the unbounded non-smooth domain defined in case (i) and for the bounded non-smooth domain defined in case (ii) of Corollary~\ref{cor:unbounded}, it is convenient to set the following notations.

\begin{definition}\label{def:truncatedworm}{\em 
We set the notations $C_\mu:=\bbR^2\setminus(-\mu,\mu)\times(-\mu,\mu)$ and $\cW_\mu':=\cW_{\chi_{C_\mu}}$, as well as $\cW'_\infty:=\cW_0$.
}
\end{definition}

\begin{remark} {\em 
The equality $\cW'_\infty=\bigcup_{\nu\in(0,+\infty)}\cW'_\nu$ holds true. 
}
\end{remark}

If we fix $z_3=1$ in the bounded non-smooth domain $\cW_\mu'$ constructed in case (ii), we recover the truncated worm domains studied in~\cite{Ki} and in subsequent literature. If we fix $z_3=1$ in the unbounded smooth domain constructed in case (iii) of Corollary~\ref{cor:unbounded}, we recover the original Diederich-Forn{\ae}ss worm domains of $\bbC^2$. Our next goal is showing that the following class of smoothly bounded pseudoconvex domains is not empty.

\begin{definition}\label{def:classofsmoothlyboundedworms}{\em 
We let $\mathscr{C}_\mu$ denote the class of all $\cW_\eta$ with
$\eta$ such that the extra assumptions (I),(II),(III) of
Proposition~\ref{prop:complextangent} and the
inequalities~\eqref{eq:pseudoconvex1},~\eqref{eq:pseudoconvex3} are
all fulfilled near $\partial\cW_\mu$.
}
\end{definition}

Our first attempt to prove that $\mathscr{C}_\mu$ is not empty leads to a negative result, which is however instructive.

\begin{proposition}
If we pick
\[\eta(t_2,t_3)=\phi(t_2)+\psi(t_3)\]
for smooth, convex and even functions $\phi,\psi$ having $\phi^{-1}(0)=\psi^{-1}(0)=[-\mu,\mu]$ and $\phi^{-1}([0,1])=[-\mu',\mu'],\psi^{-1}([0,1])=[-\xi,\xi]$ (for some $\mu',\xi>\mu$), then $\cW_\eta$ is a smoothly bounded domain in $\bbC^3$ that is {\bf not} pseudoconvex.
\end{proposition}

\begin{proof}
Under such hypotheses, we find $\eta'_2=\phi'(\log|z_2|^2), \eta'_3=\psi'(\log|z_3|^2),\eta''_{22}=\phi''(\log|z_2|^2),$ $\eta''_{33}=\psi''(\log|z_3|^2),\eta''_{23}\equiv0$. For $z=(z_1,z_2,z_3)\in\bbC\times\bbC^*\times A(0, e^{-\mu/2}, e^{\mu/2})$, we get that $\eta^\circ=\phi(\log|z_2|^2)$, whence $\eta'_3=\eta''_{33}=\eta''_{23}=0$. For all $\alpha_1,\alpha_2,\alpha_3\in\bbC$, by inspection in the proof of Theorem~\ref{thm:pseudoconvex},
\[h_{M+N}(\alpha_1{\bf p}+\alpha_2{\bf n}_2+\alpha_3{\bf n}_3,\alpha_1{\bf p}+\alpha_2{\bf n}_2+\alpha_3{\bf n}_3)=|\alpha_1|^2+\begin{pmatrix}\alpha_2&\alpha_3\end{pmatrix}C(\log|z_2|^2)\begin{pmatrix}\overline{\alpha}_2\\\overline{\alpha}_3\end{pmatrix}\,,\]
where
\[{\bf p}
=\begin{pmatrix}
1\\0\\0
\end{pmatrix}\,,\quad
{\bf n}_2
=\begin{pmatrix}
i z_1\\z_2\\0
\end{pmatrix}\,,\quad
{\bf n}_3
=\begin{pmatrix}
i z_1\\0\\z_3
\end{pmatrix}\,,\quad C=\begin{pmatrix}
\phi+\phi''&\phi+i\phi'\\
\phi-i\phi'&\phi
\end{pmatrix}\,.\]
Inequality~\eqref{eq:pseudoconvex3} is the same as $\det C(\log|z_2|^2)\geq0$. For $\varepsilon>0$, we have $\det C(\log|z_2|^2)\geq0$ for $z_2\in A(0,e^{\mu/2},e^{(\mu+\varepsilon)/2})$ if, and only if $\phi\phi''-(\phi')^2\geq0$ in $(\mu,\mu+\varepsilon)$. To the contrary, we claim that $\liminf_{\tau\to\mu^+}\left[\frac{\phi''(\tau)}{\phi(\tau)}-\left(\frac{\phi'(\tau)}{\phi(\tau)}\right)^2\right]=-\infty$ and postpone the proof of our claim to the end of the proof.

Let us now prove that $\cW_\eta$ is not pseudoconvex. Take any $z=(z_1,z_2,z_3)\in\partial\cW_\eta$ with $z_3\in A(0, e^{-\mu/2}, e^{\mu/2})$. Based on Remark~\ref{rmk:firstderivatives}, we note that the vector
\[\alpha_1{\bf p}+\alpha_2{\bf n}_2+\alpha_3{\bf n}_3=\begin{pmatrix}
\alpha_1+i z_1(\alpha_2+\alpha_3)\\z_2\alpha_2\\z_3\alpha_3
\end{pmatrix}\]
belongs to the complex tangent to $\partial\cW_\eta$ at $z$ if, and only if,
\[(\overline{z}_1 - e^{-iL})\alpha_1+(i|z_1|^2-iz_1e^{-iL})(\alpha_2+\alpha_3)+(-2 \Im (z_1 e^{-iL})+\phi'(\log|z_2|^2))\alpha_2-2 \Im (z_1 e^{-iL})\alpha_3=0\,.\]
This is the same as
\[(\overline{z}_1 - e^{-iL})\alpha_1=\zeta\,(\alpha_2+\alpha_3)-\phi'(\log|z_2|^2)\alpha_2\,,\quad\zeta:=-i|z_1|^2+i\overline{z}_1e^{iL}=-i\overline{z}_1(z_1-e^{iL})\,.\]
Now, $z\in\partial\cW_\eta$ implies $\big|\overline{z}_1 - e^{-iL}\big|^2=1-\eta^\circ=1-\phi(\log|z_2|^2)$ and $|\zeta|^2=|z_1|^2(1-\phi(\log|z_2|^2))$. Thus,
\begin{align*}
&(1-\phi(\log|z_2|^2))^{-1}\,|\alpha_1|^2=\big|\zeta(\alpha_2+\alpha_3)-\phi'(\log|z_2|^2)\alpha_2\big|^2\\
&=|\zeta-\phi'(\log|z_2|^2)\big|^2\,|\alpha_2|^2+|\zeta|^2\,|\alpha_3|^2+(\zeta-\phi'(\log|z_2|^2))\overline{\zeta}\alpha_2\overline{\alpha}_3+\zeta(\overline{\zeta}-\phi'(\log|z_2|^2))\alpha_3\overline{\alpha}_2\,.
\end{align*}
Overall,
\[h_{M+N}(\alpha_1{\bf p}+\alpha_2{\bf n}_2+\alpha_3{\bf n}_3,\alpha_1{\bf p}+\alpha_2{\bf n}_2+\alpha_3{\bf n}_3)=\begin{pmatrix}\alpha_2&\alpha_3\end{pmatrix}\left(B_{z_1}(\log|z_2|^2)+C(\log|z_2|^2)\right)\begin{pmatrix}\overline{\alpha}_2\\\overline{\alpha}_3\end{pmatrix}\,,\]
where
\[B_{z_1}=\frac1{1-\phi}\begin{pmatrix}
|\zeta-\phi'\big|^2&(\zeta-\phi')\overline{\zeta}\\
\zeta(\overline{\zeta}-\phi')&|\zeta|^2
\end{pmatrix}\,.\]
Since the diagonal elements of $B_{z_1}+C$ are non-negative, the study of pseudoconvexity of $\cW_\eta$ relies upon the study the sign of $\det(B_{z_1}+C)=\det B_{z_1}+\det C+d=0+\phi\phi''-(\phi')^2+d$, where
\begin{align*}
(1-\phi)\,d&=|\zeta-\phi'\big|^2\phi+|\zeta|^2(\phi+\phi'')-(\zeta-\phi')\overline{\zeta}(\phi-i\phi')-\zeta(\overline{\zeta}-\phi')(\phi+i\phi')\\
&=|\zeta-\phi'-\zeta\big|^2\phi+i[(\zeta-\phi')\overline{\zeta}-\zeta(\overline{\zeta}-\phi')]\phi'+|\zeta|^2\phi''\\
&=\phi(\phi')^2+i(\zeta-\overline{\zeta})(\phi')^2+|\zeta|^2\phi''\\
&=\phi(\phi')^2+(\overline{z}_1(z_1-e^{iL})+z_1(\overline{z}_1-e^{-iL}))(\phi')^2+|z_1|^2(1-\phi)\phi''\\
&=\left[\phi-2\Re(z_1e^{-iL})+2|z_1|^2\right](\phi')^2+|z_1|^2(1-\phi)\,\phi''\,.
\end{align*}
We must therefore study the sign of
\[\det(B_{z_1}+C)=\phi\phi''-(\phi')^2+\frac{\phi-2\Re(z_1e^{-iL})+2|z_1|^2}{1-\phi}(\phi')^2+|z_1|^2\phi''\,.\]
If we choose $z$ so that $z_1=\left(1-\sqrt{1-\phi}\right)e^{iL}$, then $\phi-2\Re(z_1e^{-iL})=\phi-2+2\sqrt{1-\phi}=O(\phi^2)$ and $|z_1|^2=(1-\sqrt{1-\phi})^2=O(\phi^2)$. Using our claim, we get
\[\liminf_{\tau\to\mu^+}\left[\phi(\tau)^{-2}\det\left(B_{\left(1-\sqrt{1-\phi}\right)e^{iL}}(\tau)+C(\tau)\right)\right]=\liminf_{\tau\to\mu^+}\left[\frac{\phi''(\tau)}{\phi(\tau)}-\left(\frac{\phi'(\tau)}{\phi(\tau)}\right)^2\right]=-\infty\,.\]
It follows that $B_{\left(1-\sqrt{1-\phi}\right)e^{iL}}(\tau)+C(\tau)$ cannot be positive semidefinite for all $\tau\in(\mu,\mu')$ and that $B_{z_1}(\log|z_2|^2)+C(\log|z_2|^2)$ cannot be positive semidefinite for all $z\in\partial\cW_\eta$. Thus, $\cW_\eta$ is not pseudoconvex.

We are left with proving our claim that $\liminf_{\tau\to\mu^+}\left[\frac{\phi''(\tau)}{\phi(\tau)}-\left(\frac{\phi'(\tau)}{\phi(\tau)}\right)^2\right]=-\infty$. The function $\phi$ coincides away from $\phi^{-1}(0)=[-\mu,\mu]$ with $e^f$, for some function $f$. It follows that $\phi'=e^f f'$ and $\phi''=e^f (f''+(f')^2)$, whence $\frac{\phi''(\tau)}{\phi(\tau)}-\left(\frac{\phi'(\tau)}{\phi(\tau)}\right)^2=f''$. Since $\lim_{t\to\mu^+}\phi(t)=0$, we have $\lim_{t\to\mu^+}f(t)=-\infty$, whence
\[\lim_{t\to\mu^+}\int_{t}^{\mu'} f'(\tau)d\tau=\lim_{t\to\mu^+}[f(\mu')-f(t)]=+\infty\,.\]
As a consequence, there exists a sequence $\{s_n\}_{n\in\bbN}\subset(\mu,\mu')$ with $\lim_{n\to+\infty}s_n=\mu$ such that $\lim_{n\to+\infty}f'(s_n)=+\infty$. Thus,
\[\lim_{n\to+\infty}\int_{s_n}^{\mu'} f''(\tau)d\tau=\lim_{n\to+\infty}[f'(\mu')-f'(s_n)]=-\infty\]
This implies that $\liminf_{\tau\to\mu^+}f''(\tau)=-\infty$, as desired.
\end{proof}

The computations made in the last proof motivate the next remark and the subsequent proposition.

\begin{remark} {\em 
Convexity of $\eta$ suffices to guarantee
inequality~\eqref{eq:pseudoconvex1} but not
inequality~\eqref{eq:pseudoconvex3}.
}
\end{remark}

In the case $\eta(t_2,t_3)=\phi(t_2)+\psi(t_3)$, condition~\eqref{eq:pseudoconvex3} would have required \emph{logarithmic} convexity of $\phi$, which would have in turn prevented $\phi$ from vanishing identically in $[-\mu,\mu]$.

\begin{proposition}\label{prop:logarithmic}
Assume $\eta=\chi_U e^f$ for some open subset $U$ of $\bbR^2$ and some smooth function $f:U\to\bbR$. Set $f^\circ:=f(\log|z_2|^2,\log|z_3|^2)$, $f'_{j}:=\frac{\partial f}{\partial t_j}(\log|z_2|^2,\log|z_3|^2)$ and $f''_{j,k}:=\frac{\partial^2f}{\partial t_j\partial t_k}(\log|z_2|^2,\log|z_3|^2)$. Then the equalities
\begin{align*}
\eta^\circ&=e^{f^\circ}\,,\\
\eta'_j&=\eta^\circ f'_j\,,\\
\eta''_{j,k}&=\eta^\circ(f'_jf'_k+f''_{j,k})\,,
\end{align*}
hold true in $U$ for all $j,k\in\{2,3\}$. If we set $H_f:=\begin{pmatrix}f''_{22}&f''_{23}\\f''_{23}&f''_{33}\end{pmatrix}$, $v:=\begin{pmatrix}f'_{3}\\-f'_{2}\end{pmatrix}$, $u:=\begin{pmatrix}1\\-1\end{pmatrix}$ then, for the terms appearing in inequality~\eqref{eq:pseudoconvex3}, we have
\begin{align*}
&\eta''_{22}\eta''_{33}-(\eta''_{23})^2=\det H_\eta=(\eta^\circ)^2[h_{H_f}(v,v)+\det H_f]\\
&\eta^\circ(\eta''_{22}+\eta''_{33}-2\eta''_{23})-(\eta'_2-\eta'_3)^2=\eta^\circ h_{H_\eta}(u,u)-(\eta'_2-\eta'_3)^2=(\eta^\circ)^2 h_{H_f}(u,u)\,.
\end{align*}
Therefore, in $U$ inequalities~\eqref{eq:pseudoconvex1},~\eqref{eq:pseudoconvex3} are equivalent, respectively, to
\begin{align}
&(f'_2)^2+f''_{22}+1\geq0\,,\label{eq:pseudoconvex4}\\
&h_{H_f}(v,v)+\det H_f+h_{H_f}(u,u)\geq0\,.\label{eq:pseudoconvex6}
\end{align}
\end{proposition}

\begin{proof}
We compute
\begin{align*}
\eta''_{22}\eta''_{33}-(\eta''_{23})^2&=(\eta^\circ)^2\{[(f'_2)^2+f''_{22}][(f'_3)^2+f''_{33}]-(f'_2f'_3+f''_{23})^2\}\\
&=(\eta^\circ)^2[(f'_2)^2f''_{33}+(f'_3)^2f''_{22}-2f'_2f'_3f''_{23}+f''_{22}f''_{33}-(f''_{23})^2]\\
&=(\eta^\circ)^2[v^tH_fv+\det H_f]
\end{align*}
and
\begin{align*}
\eta^\circ(\eta''_{22}+\eta''_{33}-2\eta''_{23})-(\eta'_2-\eta'_3)^2&=(\eta^\circ)^2[(f'_2)^2+f''_{22}+(f'_3)^2+f''_{33}-2f'_2f'_3-2f''_{23}-(f'_2-f'_3)^2]\\
&=(\eta^\circ)^2[f''_{22}+f''_{33}-2f''_{23}]\\
&=(\eta^\circ)^2u^tH_fu\,.\qedhere
\end{align*}
\end{proof}

The last result is useful to prove that the class $\mathscr{C}_\mu$ is not empty, thus finding examples of smoothly bounded pseudoconvex domains $\cW_\eta$. The next proposition will also be useful in the construction.

\begin{proposition}\label{prop:decomposition}
Let $U_0,\ldots,U_l$ be open subsets of $\bbR^2$, take smooth functions $\eta_0,\ldots,\eta_l:\bbR^2\to\bbR$ and set
\[\eta:=\eta_0+\ldots+\eta_l\,.\]
If each $\eta_k$ fulfils inequality~\eqref{eq:pseudoconvex1} separately, then $\eta$ itself fulfils inequality~\eqref{eq:pseudoconvex1}. This is automatically true if $\eta_0,\ldots,\eta_l$ are convex. If we assume $l=1$, i.e., $\eta:=\eta_0+\eta_1$, we can compute the quantity
\[\cP(\eta)=\det H_\eta+\eta^\circ h_{H_\eta}(u,u)-(\eta'_2-\eta'_3)^2\] appearing on the left-hand side of inequality~\eqref{eq:pseudoconvex3} as follows:
\begin{align*}
\cP(\eta)&=\cP(\eta_0)+\cP(\eta_1)+\eta_0^\circ \det H_{\eta_1}+\eta_1^\circ \det H_{\eta_0}\\
&\quad+(\eta_0)''_{22}(\eta_1)''_{33}+(\eta_1)''_{22}(\eta_0)''_{33}-2(\eta_0)''_{23}(\eta_1)''_{23}-2\left[(\eta_0)'_2-(\eta_0)'_3\right]\left[(\eta_1)'_2-(\eta_1)'_3\right]\,.
\end{align*}
\end{proposition}

\begin{proof}
The first statement follows from the equality
\[\eta^\circ+\eta''_{22}=\sum_{k=1}^l((\eta_k)^\circ+(\eta_k)''_{22})\,.\]
For the second statement, we compute
\begin{align*}
\cP(\eta)&=\det H_\eta+\eta^\circ h_{H_\eta}(u,u)-(\eta'_2-\eta'_3)^2\\
&=\left[(\eta_0)''_{22}+(\eta_1)''_{22}\right]\left[(\eta_0)''_{33}+(\eta_1)''_{33}\right]-\left[(\eta_0)''_{23}+(\eta_1)''_{23}\right]^2\\
&\quad+(\eta_0^\circ+\eta_1^\circ) \left[h_{H_{\eta_0}}(u,u)+h_{H_{\eta_1}}(u,u)\right]\\&
\quad-((\eta_0)'_2+(\eta_1)'_2-(\eta_0)'_3-(\eta_1)'_3)^2\\
&=\cP(\eta_0)+\cP(\eta_1)+(\eta_0)''_{22}(\eta_1)''_{33}+(\eta_1)''_{22}(\eta_0)''_{33}-2(\eta_0)''_{23}(\eta_1)''_{23}\\
&\quad+\eta_0^\circ \det H_{\eta_1}+\eta_1^\circ \det H_{\eta_0}-2((\eta_0)'_2-(\eta_0)'_3)((\eta_1)'_2-(\eta_1)'_3)\,.\qedhere
\end{align*}
\end{proof}

We are now ready to prove that the class $\mathscr{C}_\mu$ is not empty, constructing a class of examples of $\eta$'s such that the corresponding $\cW_\eta$ are smoothly bounded pseudoconvex domains.

\begin{theorem}\label{thm:main}
Pick $A_\pm>B_\pm\geq\sqrt{2e^\mu}$ and let $c_\pm>0$. Let us define $\eta:=\eta_++\eta_-$, where
\begin{align*}
\eta_\pm(t_2,t_3)&:=\chi_{(B_\pm^2,+\infty)}(e^{\pm t_2}+e^{\pm t_3})\,e^{f_\pm(t_2,t_3)}\,,\\
f_\pm(t_2,t_3)&:=\frac{c_\pm}{A_\pm^2-B_\pm^2}-\frac{c_\pm}{e^{\pm t_2}+e^{\pm t_3}-B_\pm^2}\,.
\end{align*}
In other words,
\begin{align*}
\eta^\circ&=\chi_{(B_+^2,+\infty)}(|z_2|^2+|z_3|^2)\,\exp\left(\frac{c_+}{A_+^2-B_+^2}\,\frac{|z_2|^2+|z_3|^2-A_+^2}{|z_2|^2+|z_3|^2-B_+^2}\right)\\
&\quad+\chi_{(B_-^2,+\infty)}(|z_2|^{-2}+|z_3|^{-2})\,\exp\left(\frac{c_-}{A_-^2-B_-^2}\,\frac{|z_2|^{-2}+|z_3|^{-2}-A_-^2}{|z_2|^{-2}+|z_3|^{-2}-B_-^2}\right)\,.
\end{align*}
Then $\cW_\eta$ is a smoothly bounded domain in $\bbC^3$ that contains $\cW_\mu'$ (and is contained in $\cW'_{\mu'}$ for $\mu'\geq2\log A_+,2\log A_-$). Moreover, there exists $C_+\geq A_+^2$ such that, for every $c_+>C_+$, there exist $C_-=C_-(c_+)\geq A_-^2$ such that, for every $c_->C_-$, the domain $\cW_\eta$ is pseudoconvex.
\end{theorem}

\begin{proof}
The functions $\eta,\eta_+,\eta_-$ are continuous because $f_\pm(t_2,t_3)\to-\infty$ as $e^{\pm t_2}+e^{\pm t_3}\to B_\pm^2$. Moreover, $\eta,\eta_+,\eta_-$ are $C^2$ because
\begin{align*}
\frac{\partial \eta_\pm}{\partial t_j}(t_2,t_3)&=\chi_{(B_\pm^2,+\infty)}(e^{\pm t_2}+e^{\pm t_3})e^{f_\pm(t_2,t_3)}\frac{\partial f_\pm}{\partial t_j}(t_2,t_3)\,,\\
\frac{\partial^2 \eta_\pm}{\partial t_j\partial t_k}(t_2,t_3)&=\chi_{(B_\pm^2,+\infty)}(e^{\pm t_2}+e^{\pm t_3})e^{f_\pm(t_2,t_3)}\left(\frac{\partial f_\pm}{\partial t_j}\frac{\partial f_\pm}{\partial t_k}+\frac{\partial^2 f_\pm}{\partial t_j\partial t_k}\right)(t_2,t_3)\,,
\end{align*}
where
\begin{align*}
\frac{\partial f_\pm}{\partial t_j}(t_2,t_3)&=\frac{\pm c_\pm e^{\pm t_j}}{(e^{\pm t_2}+e^{\pm t_3}-B_\pm^2)^2}\\
\frac{\partial^2 f_\pm}{\partial t_2^2}(t_2,t_3)&=\frac{c_\pm e^{\pm t_2}(e^{\pm t_2}+e^{\pm t_3}-B_\pm^2)^2-c_\pm e^{\pm t_2}2(e^{\pm t_2}+e^{\pm t_3}-B_\pm^2)e^{\pm t_2}}{(e^{\pm t_2}+e^{\pm t_3}-B_\pm^2)^4}\\
&=c_\pm e^{\pm t_2}\frac{-e^{\pm t_2}+e^{\pm t_3}-B_\pm^2}{(e^{\pm t_2}+e^{\pm t_3}-B_\pm^2)^3}\\
\frac{\partial^2 f_\pm}{\partial t_3^2}(t_2,t_3)&=c_\pm e^{\pm t_3}\frac{e^{\pm t_2}-e^{\pm t_3}-B_\pm^2}{(e^{\pm t_2}+e^{\pm t_3}-B_\pm^2)^3}\\
\frac{\partial^2 f_\pm}{\partial t_2\partial t_3}(t_2,t_3)&=\frac{-2c_\pm e^{\pm t_2}e^{\pm t_3}}{(e^{\pm t_2}+e^{\pm t_3}-B_\pm^2)^3}\,.
\end{align*}
Similar reasonings apply to successive derivatives and prove that $\eta,\eta_+,\eta_-$ are smooth. Clearly, the gradient of $\eta_\pm$ never vanishes where $\eta_\pm\neq0$.

The function $\eta$ vanishes identically in the (compact) intersection of the sets
\begin{align*}
(\eta_+)^{-1}(0)&=\{(t_2,t_3)\in\bbR^2:e^{t_2}+e^{t_3}\leq B_+^2\}\,,\\
(\eta_-)^{-1}(0)&=\{(t_2,t_3)\in\bbR^2:e^{-t_2}+e^{-t_3}\leq B_-^2\}\,,
\end{align*}
which includes the square $(-\mu,\mu)\times(-\mu,\mu)$ because $e^{-\mu}<e^{t_2},e^{t_3}< e^{\mu}$ implies $e^{t_2}+e^{t_3}<2e^{\mu}\leq B_+^2$ and $e^{-t_2}+e^{-t_3}<2e^{\mu}\leq B_-^2$. In particular, $\cW'_\mu\subset\cW_\eta$.
The function $\eta$ takes positive values wherever $\eta_+$ or $\eta_-$ is strictly positive. Clearly, each positive level set of $\eta_+$ is a straight line $e^{t_2}+e^{t_3}=l_+>B_+^2>2$ in the $(e^{t_2},e^{t_3})$-quadrant, while each positive level set of $\eta_- $ is a hyperbola $e^{-t_2}+e^{-t_3}=l_->B_-^2>2$ (or equivalently $e^{t_2}e^{t_3}-l_-^{-1}(e^{t_2}+e^{t_3})=0$) in the $(e^{t_2},e^{t_3})$-quadrant. This straight line and this hyperbola intersect at exactly two points, namely
\[\left(e^{t_2},e^{t_3}\right)=\left(\frac{l_+\pm\sqrt{\Delta}}{2},\frac{l_+\mp\sqrt{\Delta}}{2}\right),\quad \Delta:=l_+^2-4l_+l_-^{-1}=l_+l_-^{-1}(l_+l_--4)>0\,.\]
We point out, for future reference, that $\eta_\pm(t_2,t_3)>0$ with $e^{t_2}=e^{t_3}$ implies $\eta_\mp(t_2,t_3)=0$.
\begin{figure}[t]
\begin{center}
\includegraphics[height=7cm]{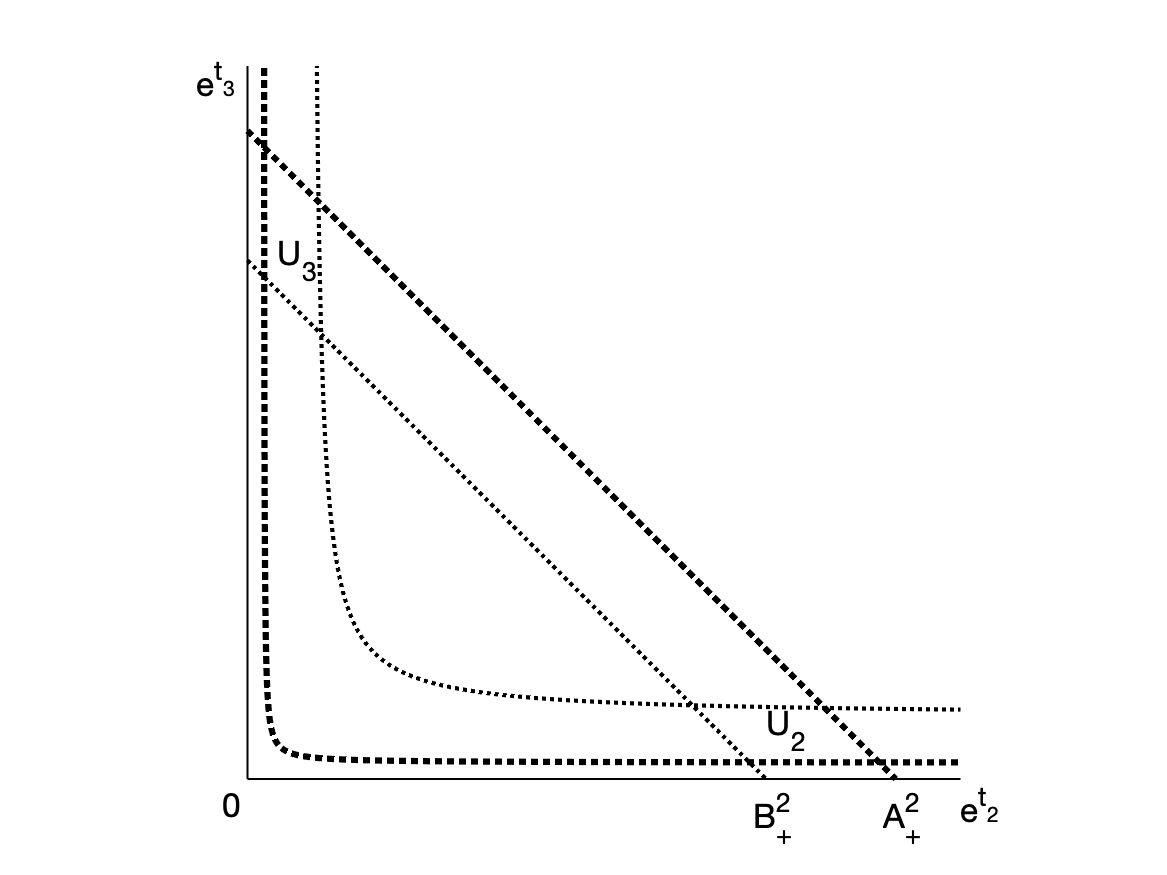}
\end{center}
\caption{The curves $e^{t_2}+e^{t_3}=A_+^2,e^{-t_2}+e^{-t_3}=A_-^2$ (in bold dots) and $e^{t_2}+e^{t_3}=B_+^2,e^{-t_2}+e^{-t_3}=B_-^2$ (in finer dots) within the $(e^{t_2},e^{t_3})$-quadrant.}
\label{fig1}
\end{figure}

The set 
\[(\eta_+)^{-1}((0,1])=\{(t_2,t_3)\in\bbR^2:B_+^2<e^{t_2}+e^{t_3}\leq A_+^2\}\,,\]
corresponds to the set between the dotted straight lines in Figure~\ref{fig1}, while the set
\[(\eta_-)^{-1}((0,1])=\{(t_2,t_3)\in\bbR^2:B_-^2<e^{-t_2}+e^{-t_3}\leq A_-^2\}\,,\]
corresponds to the set between the two dotted hyperbola branches in Figure~\ref{fig1} (bold curve included, finer curve excluded). The intersection $(\eta_+)^{-1}((0,1])\cap(\eta_-)^{-1}((0,1])$ has two connected components $U_2,U_3$, with
\begin{equation}\label{eq:connectedcomponents}
U_2\subseteq\{(t_2,t_3)\in\bbR^2:t_3<t_2\},\quad U_3\subseteq\{(t_2,t_3)\in\bbR^2:t_2<t_3\}\,.
\end{equation}
The compact set $\eta^{-1}([0,1])$ is properly included in the (compact) intersection between the sets
\begin{align*}
(\eta_+)^{-1}([0,1])&=\{(t_2,t_3)\in\bbR^2:e^{t_2}+e^{t_3}\leq A_+^2\}\,,\\
(\eta_-)^{-1}([0,1])&=\{(t_2,t_3)\in\bbR^2:e^{-t_2}+e^{-t_3}\leq A_-^2\}\,.
\end{align*}
In particular, if $\eta(t_2,t_3)\in[0,1]$, then $A_-^{-2}\leq e^{t_2},e^{t_3}\leq A_+^2$, whence $-2\log A_-<t_2,t_3<2\log A_+$. Thus, $\cW_\eta\subset\cW'_{\mu'}$ for $\mu'\geq2\log A_+,2\log A_-$.

More precisely, the simple closed curve $\eta^{-1}(1)$ bounding the compact set $\eta^{-1}([0,1])$ is the union of four arcs $\gamma_+,\gamma_-,\gamma_2,\gamma_3$, where
\[\gamma_+:\left\{\begin{array}{l}
e^{t_2}+e^{t_3}=A_+^2\\
e^{-t_2}+e^{-t_3}\leq B_-^2
\end{array}\right.\]
is part of the bold straight line in Figure~\ref{fig1},
\[\gamma_-:\left\{\begin{array}{l}
e^{-t_2}+e^{-t_3}=A_-^2\\
e^{t_2}+e^{t_3}\leq B_+^2
\end{array}\right.\]
is part of the bold hyperbola in Figure~\ref{fig1}, while the arcs
\[\gamma_j:\left\{\begin{array}{l}
e^{f^{-}(t_2,t_3)}+e^{f^{+}(t_2,t_3)}=1\\
(t_2,t_3)\in U_j
\end{array}\right.\]
(with $j\in\{2,3\}$) are not drawn in Figure~\ref{fig1}. Clearly, the gradient of $\eta$ never vanishes in $\gamma_+$, where it equals the gradient of $\eta_+$, nor in $\gamma_-$, where it equals the gradient of $\eta_-$.  Moreover, the gradient of $\eta$ never vanishes in $\gamma_2,\gamma_3$ because it never vanishes in $U_2\cup U_3$: if it did, then
\begin{align*}
e^{f_+(t_2,t_3)}\frac{c_+ e^{t_2}}{(e^{t_2}+e^{t_3}-B_+^2)^2}-e^{f_-(t_2,t_3)}\frac{c_- e^{-t_2}}{(e^{-t_2}+e^{- t_3}-B_-^2)^2}&=0
\end{align*}
and
\begin{align*}
e^{f_+(t_2,t_3)}\frac{c_+ e^{t_3}}{(e^{t_2}+e^{t_3}-B_+^2)^2}-e^{f_-(t_2,t_3)}\frac{c_- e^{-t_3}}{(e^{-t_2}+e^{- t_3}-B_-^2)^2}&=0\,,
\end{align*}
whence
\[e^{2t_2}=\frac{c_-}{c_+}\frac{e^{f_-(t_2,t_3)}}{e^{f_+(t_2,t_3)}}\frac{(e^{t_2}+e^{t_3}-B_+^2)^2}{(e^{-t_2}+e^{- t_3}-B_-^2)^2}=e^{2t_3}\]
and $t_2=t_3$. We would thus find a contradiction with the inclusions~\eqref{eq:connectedcomponents}. We have therefore proven that the gradient of $\eta$ never vanishes in $\eta^{-1}(1)$. Overall, all assumptions (I),(II),(III) are fulfilled. By Theorem~\ref{thm:pseudoconvex}, the set $\cW_\eta$ is a smoothly bounded domain in $\bbC^3$, which is pseudoconvex if, and only if, inequalities~\eqref{eq:pseudoconvex1},~\eqref{eq:pseudoconvex3} are fulfilled near $\partial\cW_\eta$. We will now prove that, for sufficiently large $c_+$, the function $\eta_+$ is convex in a neighborhood of $(\eta_+)^{-1}([0,1])$. Then we will prove an analogous statement for $\eta_-$. This will allow us to easily address inequality~\eqref{eq:pseudoconvex1}, thanks to Proposition~\ref{prop:decomposition}. Later, we will study inequality~\eqref{eq:pseudoconvex3} for $\eta$.

Omitting the sign $+$ in $f_+,c_+,A_+,B_+$ for the sake of readability, we have
\begin{align*}
f'_j&=\frac{c|z_j|^2}{(|z_2|^2+|z_3|^2-B^2)^2}\\
f''_{22}&=c|z_2|^2\frac{-|z_2|^2+|z_3|^2-B^2}{(|z_2|^2+|z_3|^2-B^2)^3}\\
f''_{33}&=c|z_3|^2\frac{|z_2|^2-|z_3|^2-B^2}{(|z_2|^2+|z_3|^2-B^2)^3}\\
f''_{23}&=c\frac{-2|z_2|^2|z_3|^2}{(|z_2|^2+|z_3|^2-B^2)^3}\,.
\end{align*}
The inequality $(\eta_+)''_{22}\geq0$ is trivial where $\eta_+^\circ$ vanishes and equivalent to $(f'_2)^2+f''_{22}\geq0$ elsewhere by Proposition~\ref{prop:logarithmic}. In order to study this last inequality, we compute
\begin{align*}
c^{-1}|z_2|^{-4}(|z_2|^2+|z_3|^2-B^2)^4\left[(f'_2)^2+f''_{22}\right]&=c+|z_2|^{-2}(-|z_2|^2+|z_3|^2-B^2)(|z_2|^2+|z_3|^2-B^2)\\
&=c-|z_2|^2+|z_2|^{-2}(|z_3|^2-B^2)^2\,.
\end{align*}
The inequality $(f'_2)^2+f''_{22}\geq0$ is certainly true when $c\geq|z_2|^2$. Choosing $c=c_+> A_+^2$ guarantees that $(\eta_+)''_{22}\geq0$ in a neighborhood of $(\eta_+^\circ)^{-1}([0,1])$. We now study the inequality $\det H_{\eta_+}\geq0$, which is trivial where $\eta_+^\circ$ vanishes and equivalent to the inequality $h_{H_f}(v,v)+\det H_f\geq0$ elsewhere by Proposition~\ref{prop:logarithmic}. Here,
\[H_f=\frac{c}{(|z_2|^2+|z_3|^2-B^2)^3}
\begin{pmatrix}
|z_2|^2(-|z_2|^2+|z_3|^2-B^2)&-2|z_2|^2|z_3|^2\\
-2|z_2|^2|z_3|^2&|z_3|^2(|z_2|^2-|z_3|^2-B^2)
\end{pmatrix}\]
and
\[v=v_+:=\begin{pmatrix}f'_{3}\\-f'_{2}\end{pmatrix}=\frac{c}{(|z_2|^2+|z_3|^2-B^2)^2}\begin{pmatrix}|z_3|^2\\-|z_2|^2\end{pmatrix}\,.\]
Recalling the notation set in formula~\eqref{eq:hermitianform}, we compute
\begin{align*}
&c^{-3}|z_2|^{-2}|z_3|^{-2}(|z_2|^2+|z_3|^2-B^2)^7h_{H_f}(v,v)\\
&=|z_2|^{-2}|z_3|^{-2}\left(c^{-1}(|z_2|^2+|z_3|^2-B^2)^2v^t\right)\left(c^{-1}(|z_2|^2+|z_3|^2-B^2)^3H_f\right)\\
&\quad\cdot\left(c^{-1}(|z_2|^2+|z_3|^2-B^2)^2\bar v\right)\\
&=|z_2|^{-2}|z_3|^{-2}(|z_3|^2,-|z_2|^2)\begin{pmatrix}
|z_2|^2|z_3|^2(-|z_2|^2+|z_3|^2-B^2)+2|z_2|^4|z_3|^2\\
-2|z_2|^2|z_3|^4-|z_2|^2|z_3|^2(|z_2|^2-|z_3|^2-B^2)
\end{pmatrix}\\
&=|z_3|^2(-|z_2|^2+|z_3|^2-B^2)+2|z_2|^2|z_3|^2+2|z_2|^2|z_3|^2+|z_2|^2(|z_2|^2-|z_3|^2-B^2)\\
&=|z_3|^2(-|z_2|^2+|z_3|^2-B^2+2|z_2|^2)+|z_2|^2(|z_2|^2-|z_3|^2-B^2+2|z_3|^2)\\
&=(|z_2|^2+|z_3|^2)(|z_2|^2+|z_3|^2-B^2)\,,
\end{align*}
whence
\[h_{H_f}(v,v)=c^3|z_2|^2|z_3|^2\frac{|z_2|^2+|z_3|^2}{(|z_2|^2+|z_3|^2-B^2)^6}=c^3|z_2|^2|z_3|^2\frac{t^2}{(t^2-B^2)^6}\,,\]
where $t=t_+:=\sqrt{|z_2|^2+|z_3|^2}$. Moreover,
\begin{align*}
&c^{-2}|z_2|^{-2}|z_3|^{-2}(|z_2|^2+|z_3|^2-B^2)^6\det H_f\\
&=|z_2|^{-2}|z_3|^{-2}\det\left(c^{-1}(|z_2|^2+|z_3|^2-B^2)^3H_f\right)\\
&=|z_2|^{-2}|z_3|^{-2}\det\begin{pmatrix}
|z_2|^2(-|z_2|^2+|z_3|^2-B^2)&-2|z_2|^2|z_3|^2\\
-2|z_2|^2|z_3|^2&|z_3|^2(|z_2|^2-|z_3|^2-B^2)
\end{pmatrix}\\
&=(-|z_2|^2+|z_3|^2-B^2)(|z_2|^2-|z_3|^2-B^2)-4|z_2|^2|z_3|^2\\
&=B^4-|z_2|^4-|z_3|^4-2|z_2|^2|z_3|^2\\
&=B^4-(|z_2|^2+|z_3|^2)^2\\
&=-(|z_2|^2+|z_3|^2+B^2)(|z_2|^2+|z_3|^2-B^2)\,,
\end{align*}
whence
\[\det H_f=-c^2|z_2|^2|z_3|^2\frac{|z_2|^2+|z_3|^2+B^2}{(|z_2|^2+|z_3|^2-B^2)^5}=-c^2|z_2|^2|z_3|^2\frac{t^2+B^2}{(t^2-B^2)^5}\,.\]
We conclude that
\begin{align*}
c^{-2}|z_2|^{-2}|z_3|^{-2}t^{-2}(t^2-B^2)^6\left[h_{H_f}(v,v)+\det H_f\right]&=c-t^{-2}(t^2+B^2)(t^2-B^2)\\
&=c-t^{-2}(t^4-B^4)\\
&=c-t^2+\frac{B^4}{t^2}\,,
\end{align*}
is non-negative when $c\geq t^2=|z_2|^2+|z_3|^2$. Therefore, choosing $c=c_+> A_+^2$ guarantees not only $(\eta_+)''_{22}\geq0$, but also $\det H_{\eta_+}\geq0$ in a neighborhood of $(\eta_+^\circ)^{-1}([0,1])$, whence the convexity of $\eta_+$ in a neighborhood of $\eta_+^{-1}([0,1])$. For future reference, we now compute the quantity $\cP(\eta_+)=\det H_{\eta_+}+\eta^\circ h_{H_{\eta_+}}(u,u)-((\eta_+)'_2-(\eta_+)'_3)^2$, defined in Proposition~\ref{prop:decomposition}. Here, $u:=(1,-1)^t$. By Proposition~\ref{prop:logarithmic}, $\cP(\eta_+)=(\eta_+^\circ)^2[h_{H_f}(v,v)+\det H_f+h_{H_f}(u,u)]$ wherever $\eta_+^\circ$ does not vanish. We compute
\begin{align*}
&c^{-1}(|z_2|^2+|z_3|^2-B^2)^3h_{H_f}(u,u)\\
&=u^t\left(c^{-1}(|z_2|^2+|z_3|^2-B^2)^3H_f\right)\bar u\\
&=(1,-1)\begin{pmatrix}
|z_2|^2(-|z_2|^2+|z_3|^2-B^2)+2|z_2|^2|z_3|^2\\
-2|z_2|^2|z_3|^2-|z_3|^2(|z_2|^2-|z_3|^2-B^2)
\end{pmatrix}\\
&=|z_2|^2(-|z_2|^2+|z_3|^2-B^2)+|z_3|^2(|z_2|^2-|z_3|^2-B^2)+4|z_2|^2|z_3|^2\\
&=-|z_2|^4-|z_3|^4+6|z_2|^2|z_3|^2-(|z_2|^2+|z_3|^2)B^2\\
&=8|z_2|^2|z_3|^2-(|z_2|^2+|z_3|^2)(|z_2|^2+|z_3|^2+B^2)\,,
\end{align*}
whence
\[h_{H_f}(u,u)=c\,\frac{8|z_2|^2|z_3|^2-(|z_2|^2+|z_3|^2)(|z_2|^2+|z_3|^2+B^2)}{(|z_2|^2+|z_3|^2-B^2)^3}=c\,\frac{8|z_2|^2|z_3|^2-t^2(t^2+B^2)}{(t^2-B^2)^3}\,.\]
We conclude that
\begin{align*}
&c^{-1}|z_2|^{-2}|z_3|^{-2}t^{-2}(t^2-B^2)^6\left[h_{H_f}(v,v)+\det H_f +h_{H_f}(u,u)\right]\\
&=c^2+c\left(-t^2+\frac{B^4}{t^2}\right)+[8t^{-2}-|z_2|^{-2}|z_3|^{-2}(t^2+B^2)](t^2-B^2)^3\\
&=c^2+c\left(-t^2+\frac{B^4}{t^2}\right)-|z_2|^{-2}|z_3|^{-2}(t^4-B^4)(t^2-B^2)^2+8t^{-2}(t^2-B^2)^3\\
&\geq c^2+c\left(-t^2+\frac{B^4}{t^2}\right)-|z_2|^{-2}|z_3|^{-2}(t^4-B^4)(t^2-B^2)^2\,,
\end{align*}
where the last expression is non-negative for
\begin{align*}
c=c_+&\geq\left(t_+^2-\frac{B_+^4}{t_+^2}\right)\left(\frac12+\frac12\sqrt{1+4|z_2|^{-2}|z_3|^{-2}t_+^4(t_+^2-B_+^2)(t_+^2+B_+^2)^{-1}}\right)\,.
\end{align*}

Similarly, we can prove that choosing $c_-> A_-^2$ guarantees the convexity of $\eta_-$ in a neighborhood of $\eta_-^{-1}([0,1])$. Indeed, omitting the sign $-$ in $f_-,c_-,A_-,B_-$ for the sake of readability, we have
\begin{align*}
f'_j&=\frac{-c|z_j|^{-2}}{(|z_2|^{-2}+|z_3|^{-2}-B^2)^2}\\
f''_{22}&=c|z_2|^{-2}\frac{-|z_2|^{-2}+|z_3|^{-2}-B^2}{(|z_2|^{-2}+|z_3|^{-2}-B^2)^3}\\
f''_{33}&=c|z_3|^{-2}\frac{|z_2|^{-2}-|z_3|^{-2}-B^2}{(|z_2|^{-2}+|z_3|^{-2}-B^2)^3}\\
f''_{23}&=c\frac{-2|z_2|^{-2}|z_3|^{-2}}{(|z_2|^{-2}+|z_3|^{-2}-B^2)^3}\,.
\end{align*}
The inequality $(\eta_-)''_{22}\geq0$ holds where $\eta_-^\circ$ does not vanish if, and only if, $(f'_2)^2+f''_{22}\geq0$, where
\begin{align*}
&c^{-1}|z_2|^4(|z_2|^{-2}+|z_3|^{-2}-B^2)^4\left[(f'_2)^2+f''_{22}\right]\\
&=c+|z_2|^2(-|z_2|^{-2}+|z_3|^{-2}-B^2)(|z_2|^{-2}+|z_3|^{-2}-B^2)\\
&=c-|z_2|^{-2}+|z_2|^2(|z_3|^{-2}-B^2)^2\,.
\end{align*}
The last inequality is certainly true when $c\geq|z_2|^{-2}$. On the other hand, $\det H_{\eta_-}\geq0$ where $\eta_-^\circ$ does not vanish if, and only if, (after setting $t=t_-:=\sqrt{|z_2|^{-2}+|z_3|^{-2}}$)
\[0\leq c^{-2}|z_2|^2|z_3|^2t^{-2}(t^2-B^2)^6\left[h_{H_f}(v,v)+\det H_f\right]= c-t^2+\frac{B^4}{t^2}\,.\]
The last inequality is certainly true when $c\geq t^2=|z_2|^{-2}+|z_3|^{-2}$. A suitable choice is therefore $c=c_-> A_-^2$, as announced. For future reference, we now compute the quantity $\cP(\eta_-)=\det H_{\eta_-}+\eta^\circ h_{H_{\eta_-}}(u,u)-((\eta_-)'_2-(\eta_-)'_3)^2=(\eta_-^\circ)^2[h_{H_f}(v,v)+\det H_f+h_{H_f}(u,u)]$. We have
\begin{align*}
&c^{-1}|z_2|^2|z_3|^2t^{-2}(t^2-B^2)^6\left[h_{H_f}(v,v)+\det H_f +h_{H_f}(u,u)\right]\\
&=c^2+c\left(-t^2+\frac{B^4}{t^2}\right)-|z_2|^2|z_3|^2(t^4-B^4)(t^2-B^2)^2+8t^{-2}(t^2-B^2)^3\\
&\geq c^2+c\left(-t^2+\frac{B^4}{t^2}\right)-|z_2|^2|z_3|^2(t^4-B^4)(t^2-B^2)^2\,,
\end{align*}
where the last expression is non-negative for
\begin{align*}
c=c_-&\geq\left(t_-^2-\frac{B_-^4}{t_-^2}\right)\left(\frac12+\frac12\sqrt{1+4|z_2|^2|z_3|^2t_-^4(t_-^2-B_-^2)(t_-^2+B_-^2)^{-1}}\right)\,.
\end{align*}

Consider now the compact set
\begin{align*}
K&:=(\eta^\circ)^{-1}([0,1])\subseteq(\eta_+^\circ)^{-1}([0,1])\cap(\eta_-^\circ)^{-1}([0,1])\\
&=\{(z_2,z_3)\in\bbC^2:t_+\leq A_+^2,t_-\leq A_-^2\}\,.
\end{align*}
For future reference, we point out that $(z_2,z_3)\in K$ implies $A_-^{-1}\leq|z_2|,|z_3|\leq A_+$ and that $\frac{\sqrt{2}}{B_-}\leq|z_2|,|z_3|\leq\frac{B_+}{\sqrt{2}}$ implies $\eta^\circ(z)=0$. We established that $\eta_+,\eta_-$ are both convex in a neighborhood of $K$ when $c_+>A_+^2$ and $c_->A_-^2$. Proposition~\ref{prop:decomposition} guarantees that, when $c_+>A_+^2$ and $c_->A_-^2$, then inequality~\eqref{eq:pseudoconvex1} is fulfilled in a neighborhood of $K$ for $\eta=\eta_++\eta_-$.  Our final aim is studying inequality~\eqref{eq:pseudoconvex3} for $\eta=\eta_++\eta_-$. According to Proposition~\ref{prop:decomposition}, inequality~\eqref{eq:pseudoconvex3} is the same as
\begin{align*}
0&\leq\cP(\eta)=\cP(\eta_+)+\cP(\eta_-)+\eta_+^\circ \det H_{\eta_-}+\eta_-^\circ \det H_{\eta_+}\\
&\quad+(\eta_+)''_{22}(\eta_-)''_{33}+(\eta_-)''_{22}(\eta_+)''_{33}-2(\eta_+)''_{23}(\eta_-)''_{23}-2((\eta_+)'_2-(\eta_+)'_3)((\eta_-)'_2-(\eta_-)'_3)\,.
\end{align*}
We already established that $\det H_{\eta_+}\geq0$ in a neighborhood of $(\eta_+^\circ)^{-1}([0,1])$ when $c_+>A_+^2$ and that $\det H_{\eta_-}\geq0$ in a neighborhood of $(\eta_-^\circ)^{-1}([0,1])$ when $c_->A_-^2$. We also know that $\cP(\eta_+),\cP(\eta_-)\geq0$ in a neighborhood of $K$ if we choose $c_+>\cM_+$ and $c_->\cM_-$, where
\begin{align*}
\cM_+&:=\max_{(z_2,z_3)\in K}\left(t_+^2-\frac{B_+^4}{t_+^2}\right)\left(\frac12+\frac12\sqrt{1+4|z_2|^{-2}|z_3|^{-2}t_+^4(t_+^2-B_+^2)(t_+^2+B_+^2)^{-1}}\right)\,,\\
\cM_-&:=\max_{(z_2,z_3)\in K}\left(t_-^2-\frac{B_-^4}{t_-^2}\right)\left(\frac12+\frac12\sqrt{1+4|z_2|^2|z_3|^2t_-^4(t_-^2-B_-^2)(t_-^2+B_-^2)^{-1}}\right)\,.
\end{align*}
We are left with studying the quantity
\[\cQ(\eta):=(\eta_+)''_{22}(\eta_-)''_{33}+(\eta_-)''_{22}(\eta_+)''_{33}-2(\eta_+)''_{23}(\eta_-)''_{23}-2\left[(\eta_+)'_2-(\eta_+)'_3\right]\left[(\eta_-)'_2-(\eta_-)'_3\right]\,.\]
We have $\cQ(\eta)=0$ wherever $\eta_+^\circ=0$ or $\eta_-^\circ=0$. Moreover,
\begin{align*}
&(\eta^\circ)^{-2}\cQ(\eta)\\
&=\left[((f_+)'_2)^2+(f_+)''_{22}\right]\left[((f_-)'_3)^2+(f_-)''_{33}\right]+\left[((f_-)'_2)^2+(f_-)''_{22}\right]\left[((f_+)'_3)^2+(f_+)''_{33}\right]\\
&\quad-2\left[(f_+)'_2(f_+)'_3+(f_+)''_{23}\right]\left[(f_-)'_2(f_-)'_3+(f_-)''_{23}\right]-2\left[(f_+)'_2-(f_+)'_3\right]\left[(f_-)'_2-(f_-)'_3\right]
\end{align*}
at each point $(z_2,z_3)$ where neither $\eta_+^\circ$ nor $\eta_-^\circ$ vanishes (which implies $|z_2|\neq|z_3|$), i.e., at each point $(z_2,z_3)$ in the disjoint union $V_2\cup V_3$ with
\begin{align*}
V_2&:\left\{\begin{array}{l}
t_+=|z_2|^2+|z_3|^2>B_+\\
t_-=|z_2|^{-2}+|z_3|^{-2}>B_-\\
|z_3|<|z_2|\\
\end{array}\right.\,,&
V_3&:\left\{\begin{array}{l}
t_+=|z_2|^2+|z_3|^2>B_+\\
t_-=|z_2|^{-2}+|z_3|^{-2}>B_-\\
|z_2|<|z_3|\\
\end{array}\right.\,.
\end{align*}
We already know that
\begin{align*}
c_+^{-1}(t_+^2-B_+^2)^2(f_+)'_2&=|z_2|^2\,,\\
c_-^{-1}(t_-^2-B_-^2)^2(f_-)'_2&=-|z_2|^{-2}\,,\\
c_+^{-1}(t_+^2-B_+^2)^4\left[((f_+)'_2)^2+(f_+)''_{22}\right]
&\geq|z_2|^4(c_+-|z_2|^2)\,,\\
c_-^{-1}(t_-^2-B_-^2)^4\left[((f_-)'_2)^2+(f_-)''_{22}\right]
&\geq|z_2|^{-4}(c_--|z_2|^{-2})\,.
\end{align*}
Similarly,
\begin{align*}
c_+^{-1}(t_+^2-B_+^2)^2(f_+)'_3&=|z_3|^2\,,\\
c_-^{-1}(t_-^2-B_-^2)^2(f_-)'_3&=-|z_3|^{-2}\,,\\
c_+^{-1}(t_+^2-B_+^2)^4\left[((f_+)'_3)^2+(f_+)''_{33}\right]
&\geq|z_3|^4(c_+-|z_3|^2)\,,\\
c_-^{-1}(t_-^2-B_-^2)^4\left[((f_-)'_3)^2+(f_-)''_{33}\right]
&\geq|z_3|^{-4}(c_--|z_3|^{-2})\,.\\
c_+^{-1}(t_+^2-B_+^2)^4\left[(f_+)'_2(f_+)'_3+(f_+)''_{23}\right]&=|z_2|^2|z_3|^2\left[c_+-2(t_+^2-B_+^2)\right]\,,\\
c_-^{-1}(t_-^2-B_-^2)^4\left[(f_-)'_2(f_-)'_3+(f_-)''_{23}\right]&=|z_2|^{-2}|z_3|^{-2}\left[c_--2(t_-^2-B_-^2)\right]\,.
\end{align*}
Let us set $x:=|z_2|^2|z_3|^{-2}$. If $(z_2,z_3)\in V_2$, then $|z_2|^2>\frac{B_+^2}{2}$ and $|z_3|^{-2}>\frac{B_-^2}{2}$, whence $x>\frac{B_+^2B_-^2}{4}>1$; if $(z_2,z_3)\in V_3$, then $|z_3|^2>\frac{B_+^2}{2}$ and $|z_2|^{-2}>\frac{B_-^2}{2}$, whence $x<\frac{4}{B_+^2B_-^2}<1$. We also know that $(z_2,z_3)\in K$ implies $\frac1{A_+^2A_-^2}\leq x\leq A_+^2A_-^2$. For $(z_2,z_3)\in V_2\cup V_3$, we have
\begin{align*}
&c_+^{-1}c_-^{-1}(t_+^2-B_+^2)^4(t_-^2-B_-^2)^4(\eta^\circ)^{-2}\cQ(\eta)\\
&\geq|z_2|^4|z_3|^{-4}(c_+-|z_2|^2)(c_--|z_3|^{-2})+|z_2|^{-4}|z_3|^4(c_--|z_2|^{-2})(c_+-|z_3|^2)\\
&\quad-2\left[c_+-2(t_+^2-B_+^2)\right]\left[c_--2(t_-^2-B_-^2)\right]\\
&\quad-2(t_+^2-B_+^2)^2(t_-^2-B_-^2)^2\left(|z_2|^2-|z_3|^2\right)\left(-|z_2|^{-2}+|z_3|^{-2}\right)\\
&=c_+c_-(x^2+x^{-2}-2)+c_+\left[-x^2|z_3|^{-2}-x^{-2}|z_2|^{-2}+4(t_-^2-B_-^2)\right]+\\
&\quad+c_-\left[-x^2|z_2|^2-x^{-2}|z_3|^2+4(t_+^2-B_+^2)\right]+x^3+x^{-3}-8(t_+^2-B_+^2)(t_-^2-B_-^2)\\
&\quad-2(t_+^2-B_+^2)(t_-^2-B_-^2)(x+x^{-1}-2)\\
&\geq\alpha_{11}c_+c_-+\alpha_{10}c_++\alpha_{01}c_-+\alpha_{00}\,,
\end{align*}
with 
\begin{align*}
\alpha_{11}&:=(x-x^{-1})^2>0\\
\alpha_{10}&:=-x^2|z_3|^{-2}-x^{-2}|z_2|^{-2}\\
\alpha_{01}&:=-x^2|z_2|^2-x^{-2}|z_3|^2\\
\alpha_{00}&:=-2(t_+^2-B_+^2)(t_-^2-B_-^2)\left[(x^{1/2}-x^{-1/2})^2+4(t_+^2-B_+^2)(t_-^2-B_-^2)\right]\,.
\end{align*}
\begin{figure}[t]
\begin{center}
\includegraphics[height=7cm]{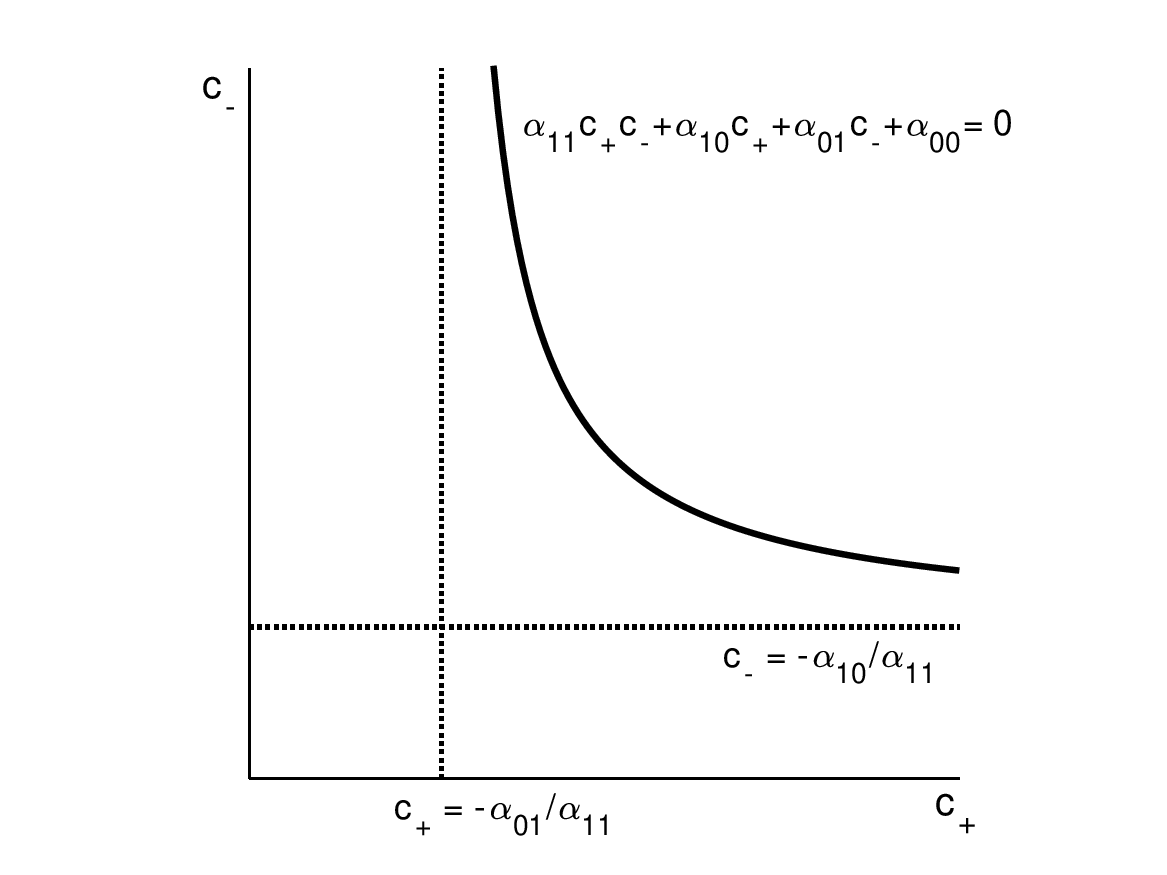}
\end{center}
\caption{The right branch of the rectangular hyperbola $\alpha_{11}c_+c_-+\alpha_{10}c_++\alpha_{01}c_-+\alpha_{00}=0$.}
\label{fig2}
\end{figure}
Therefore, $\cQ(\eta)\geq0$ in $V_2\cup V_3$ if $(c_+,c_-)$ lies to the right of the right branch of the rectangular hyperbola $\alpha_{11}c_+c_-+\alpha_{10}c_++\alpha_{01}c_-+\alpha_{00}=0$ depicted in Figure~\ref{fig2}. This happens exactly when the two following inequalities are fulfilled:
\begin{align*}
c_+&\geq-\frac{\alpha_{01}}{\alpha_{11}}=\frac{x^2|z_2|^2+x^{-2}|z_3|^2}{(x-x^{-1})^2}=\frac{x^4|z_2|^2+|z_3|^2}{(x^2-1)^2}\,,\\
c_-&\geq\frac{-\alpha_{10}c_+-\alpha_{00}}{\alpha_{11}c_++\alpha_{01}}\,.
\end{align*}
We can set
\[\cN_+:=\max_{(z_2,z_2)\in K}\left(\frac{x^4|z_2|^2+|z_3|^2}{(x^2-1)^2}\right)\,,\]
pick a $c_+$ larger than $A_+^2,\cM_+$ and $\cN_+$, set
\[\cN_-=\cN_-(c_+):=\max_{(z_2,z_2)\in K}\frac{-\alpha_{10}c_+-\alpha_{00}}{\alpha_{11}c_++\alpha_{01}}\,,\]
and pick a $c_-$ larger than $A_-^2,\cM_-$ and $\cN_-$ to conclude that $\cW_\eta$ is a smoothly bounded pseudoconvex domain.
\end{proof}

By inspection in the last proof, the following property could be added to the statement: there exists $C_-\geq A_-^2$ such that, for every $c_->C_-$, there exist $C_+=C_+(c_-)\geq A_+^2$ such that, for every $c_+>C_+$, the domain $\cW_\eta$ is pseudoconvex.\\\\


\section{Study of the smoothly bounded pseudoconvex worm domains in $\mathscr{C}_\mu$}\label{sec:study}

This section is devoted to the study of the class $\mathscr{C}_\mu$ of smoothly bounded pseudoconvex domains constructed in Definition~\ref{def:classofsmoothlyboundedworms} and proven nonempty in Theorem~\ref{thm:main}. We make the following remark, which involves the bounded non-smooth domain $\cW'_\mu$ and the unbounded domain $\cW'_\infty$ defined in Definition~\ref{def:truncatedworm}.

\begin{remark} {\em
Let $\Omega\in\mathscr{C}_\mu$. Then
\begin{equation}\label{eq:worminclusions}
\cW'_\mu\subset\Omega\subset\cW'_{\mu'}\subset\cW'_\infty
\end{equation}
for sufficiently large $\mu'>\mu$. Namely, if $\Omega=\cW_\eta$, it
suffices for the open disk of radius $\mu'$ centered at the origin in
$\bbR^2$ to include the set $\eta^{-1}([0,1))$. }
\end{remark}

It is useful to define the following functions on $\cW'_\infty$.

\begin{definition}\label{def:logarithm} {\em
For all $z=(z_1,z_2,z_3)\in\cW'_\infty$ and all $\kappa\in\bbC$, we define
\begin{align*}
\ell(z)&:=\log\left(z_1e^{-i \log |z_2 z_3|^2}\right)+i \log |z_2 z_3|^2\,,\\
E_\kappa(z)&:=\exp\left(\kappa \ell(z)\right)\,,
\end{align*}
where $\log$ denotes the principal branch of logarithm on $\bbC\setminus(-\infty,0]$.
}
\end{definition}

We can make the following remark and prove the subsequent proposition concerning the Bergman space $A^2(\Omega)$ of any $\Omega\in\mathscr{C}_\mu$.

\begin{remark} {\em
Let $z\in\cW'_\infty$. If $\kappa\in\bbZ$, then $E_\kappa(z)=z_1^\kappa$. If $\kappa\in\bbR$, then $|E_\kappa(z)|=|z_1|^\kappa$. If $\kappa\in\bbC$ and $z\in\cW'_{\frac\pi2}$, then $E_\kappa(z)=z_1^\kappa:=\exp(\kappa\log(z_1))$.
}
\end{remark}

\begin{proposition}\label{logarithmandpowers}
The functions $\ell,E_\kappa:\cW'_\infty\to\bbC$ are holomorphic and locally constant in $z_2$ and in $z_3$. Moreover, $\frac{\partial\ell}{\partial z_1}(z)=\frac1{z_1}$ and $\frac{\partial E_\kappa}{\partial z_1}=\kappa\,E_{\kappa-1}$. Finally, the following are equivalent:
\begin{enumerate}[(i)]
\item $\Re(\kappa)>-1$;
\item $E_\kappa(z)$ belongs to $A^2(\cW'_\mu)$ for some $\mu>0$;
\item $E_\kappa(z)z_2^jz_3^k$ belongs to $A^2(\cW'_\mu)$ for all $\mu>0$ and all $j,k\in\bbZ$;
\item $E_\kappa(z)z_2^jz_3^k$ belongs to $A^2(\Omega)$ for all $\mu>0$, for all $\Omega\in\mathscr{C}_\mu$ and for all $j,k\in\bbZ$.
\end{enumerate}
\end{proposition}

\begin{proof}
The function $\ell$ (whence all functions $E_\kappa$ with $\kappa\in\bbC$) is holomorphic because
\begin{align*}
\frac{\partial\ell}{\partial\bar z_1}(z)&=\frac{0}{z_1e^{-i \log |z_2 z_3|^2}}+0\equiv0\,,\\
\frac{\partial\ell}{\partial\bar z_2}(z)&=\frac{z_1e^{-i \log |z_2 z_3|^2}(-i)}{z_1e^{-i \log |z_2 z_3|^2}}\frac{\partial\log |z_2 z_3|^2}{\partial\bar z_2}+i \frac{\partial\log |z_2 z_3|^2}{\partial\bar z_2}\equiv0\,,\\
\frac{\partial\ell}{\partial\bar z_3}(z)&=\frac{z_1e^{-i \log |z_2 z_3|^2}(-i)}{z_1e^{-i \log |z_2 z_3|^2}}\frac{\partial\log |z_2 z_3|^2}{\partial\bar z_3}+i \frac{\partial\log |z_2 z_3|^2}{\partial\bar z_3}\equiv0\,.
\end{align*}
Analogous computations prove that $\frac{\partial\ell}{\partial z_2}\equiv0\equiv\frac{\partial\ell}{\partial z_3}$, whence $\ell$ and $E_\kappa$ with $\kappa\in\bbC$ are locally constant in $z_2$ and in $z_3$. Moreover,
\[\frac{\partial\ell}{\partial z_1}(z)=\frac{e^{-i \log |z_2 z_3|^2}}{z_1e^{-i \log |z_2 z_3|^2}}+0=\frac1{z_1}\,.\]
It easily follows that $E_\kappa=e^{\kappa\ell}$ has $\frac{\partial E_\kappa}{\partial z_1}=\kappa\,E_{\kappa-1}$.

Now let us study $L^2$-integrability. For any $a,b\in\bbR,j,k\in\bbZ,\mu>0$, we remark that
\begin{align*}
&\|E_{a+ib}z_2^jz_3^k\|_{A^2(\cW'_\mu)}^2 
=\int_{\cW'_\mu}\left\vert E_{a+ib}(z) z_2^jz_3^k\right\vert^2\, dV(z)\\ 
&= \int_{-\mu< \log|z_2|^2<\mu}\int_{-\mu< \log|z_3|^2<\mu}\int_{\bigtriangleup(e^{i\log|z_2z_3|^2},1)} |z_1|^{2a} |z_2|^{2j} |z_3|^{2k}\\
&\quad\cdot\exp \big\{-2b[\arg(z_1e^{-i\log|z_2z_3|^2}) + \log|z_2z_3|^2]\big\} \, dV(z_1)dV(z_2)dV(z_3)  \\ 
&= \int_{-\mu< \log|z_2|^2<\mu}\int_{-\mu< \log|z_3|^2<\mu}\int_{\bigtriangleup(1,1)} |\zeta|^{2a}|z_2|^{2j}|z_3|^{2k}\\
&\quad\cdot\exp\big\{-2b(\arg(\zeta) + \log|z_2z_3|^2)\big\}\, dV(\zeta)\,dV(z_2)\,dV(z_3)\\ 
&= 4\pi^2 \int_{e^{-\mu/2}}^{e^{\mu/2}}\int_{e^{-\mu/2}}^{e^{\mu/2}}\int_{-\frac \pi 2}^{\frac \pi 2}\int_0^{2\cos\theta} r^{2a+1} \rho_2^{2j+1} \rho_3^{2k+1} e^{-2b(\theta+ \log (\rho_2\rho_3)^2)}\, dr\, d\theta\,{d\rho_2}\,{d\rho_3}\\ 
&= \pi^2 \int_{-\frac \pi 2}^{\frac \pi 2}\int_{\frac{\theta}2-\mu}^{\frac{\theta}2+\mu}\int_{\frac{\theta}2-\mu}^{\frac{\theta}2+\mu}\int_{-\infty}^{\log(2\cos\theta)} e^{2(a+1)s} e^{(x_2-\frac{\theta}2)(j+1)} e^{(x_3-\frac{\theta}2)(k+1)} e^{-2b(x_2+x_3)}\, ds\,dx_2\,dx_3\,d\theta \\ 
&= \pi^2 \int_{-\frac \pi 2}^{\frac \pi 2}\int_{\frac{\theta}2-\mu}^{\frac{\theta}2+\mu}\int_{\frac{\theta}2-\mu}^{\frac{\theta}2+\mu}\int_{-\infty}^{\log(2\cos\theta)}e^{2(a+1)s} ds\,
e^{x_2(j+1-2b)}\,dx_2 e^{x_3(k+1-2b)}\,dx_3\, e^{-\frac{\theta}2(j+k+2)}\,d\theta
\end{align*}
is finite if, and only if, $a>-1$. For the fifth equality, we applied the change of coordinates $(r,\rho_2,\rho_3)=\big(e^s,e^{\frac{x_2}2-\frac{\theta}4},e^{\frac{x_3}2-\frac{\theta}4}\big)$. It follows at once that properties (i),(ii),(iii) are mutually equivalent. But the chain of inclusions~\eqref{eq:worminclusions} guarantees that (iii) and (iv) are mutually equivalent. The proof is therefore complete.
\end{proof}

We are now ready to prove that, for sufficiently large $\mu$ and for $\Omega\in\mathscr{C}_\mu$, the Nebenh\"ulle of $\Omega$ is nontrivial. We actually prove a bit more about the space $A^2(\overline{\Omega})$, defined as the closure within the Bergman space $A^2(\Omega)$ of the subspace of those elements of $A^2(\Omega)$ that extend to holomorphic functions on neighborhoods of $\overline{\Omega}$.

\begin{theorem}\label{non-density}	  
Let $\mu>4\pi$. If $f\in A^2(\overline{\cW_\mu'})$, then $f$ admits a holomorphic extension to the domain
\begin{equation}\label{w-hat}
\widehat{\cW}_\mu':=\cW_\mu'\cup\bigcup_{-\frac\mu2<a<\frac\mu2-2\pi} \big\{(z_1,z_2,z_3):\,|z_1-e^{2ia}| <1,a<\log|z_j|^2<a+2\pi\textrm{\ for\ }j\in\{2,3\}\big\}.
\end{equation}
As a consequence, $A^2(\overline{\cW_\mu'})\subsetneq A^2(\cW_\mu')$. 

The same conclusions hold true with any $\Omega\in\mathscr{C}_\mu$ in place of $\cW_\mu'$.
\end{theorem}
 
\begin{proof}
First suppose $f$ to be a holomorphic function on a domain $Y\supset\overline{\cW_\mu'}$. Take $a\in\left[-\frac\mu2,\frac\mu2-2\pi\right]$ and consider the annulus
\[\mathscr{A}_a:=A\big(0,e^{\frac{a}2},e^{\frac{a}2+\pi}\big)=\left\{\zeta \in \bbC : a<\log|\zeta|^2<a+2\pi\right\}\,,\]
whose oriented boundary $\partial\mathscr{A}_a$ consists of the circles $\log|\zeta|^2=a$ and $\log|\zeta|^2=a+2\pi$. Consider also  the unit disk centered at $e^{2ia}$, namely $\bigtriangleup_a=\bigtriangleup(e^{2ia},1)$. Since $2a\in[-\mu,\mu-4\pi]$, both of the following compact subsets of $\bbC^3$ are contained in $\overline{\cW_\mu'}$:
\[\{0\}\times\overline{\mathscr{A}}_a\times\overline{\mathscr{A}}_a=\big\{(0,z_2,z_3):\, a \leq \log|z_j|^2 \leq a+2\pi\textrm{\ for\ }j\in\{2,3\}\big\}\]
and
\begin{align*}
\overline{\bigtriangleup}_a\times\partial\mathscr{A}_a\times\partial\mathscr{A}_a=\big\{ (z_1,z_2,z_3):\, |z_1-e^{2ia}| \leq 1,\log|z_j|^2\in\{a,a+2\pi\}\textrm{\ for\ }j\in\{2,3\}\big\}\,.
\end{align*}
Set
\begin{equation}\label{eq:cauchyextension}
F_a(z_1,z_2,z_3) := \frac{1}{(2\pi i)^2} \int_{\partial\mathscr{A}_a}\int_{\partial\mathscr{A}_a} \frac{f(z_1,\zeta_2,\zeta_3)}{(\zeta_2-z_2)(\zeta_3-z_3)}\, d\zeta_2 d\zeta_3\,.
\end{equation}
This defines a function $F_a$ that is holomorphic in a neighborhood of the set
\[\overline{\bigtriangleup}_a\times\mathscr{A}_a\times\mathscr{A}_a=\big\{(z_1,z_2,z_3):\,|z_1-e^{2ia}|\leq1,a<\log|z_j|^2<a+2\pi\textrm{\ for\ }j\in\{2,3\}\big\}\]
and coincides with $f$ in a neighborhood of $\{0\}\times\mathscr{A}_a\times\mathscr{A}_a$ by Cauchy's integral formula. Thus, $F_a$ coincides with $f$ in $\overline{\cW_\mu'}\cap(\overline{\bigtriangleup}_a\times\mathscr{A}_a\times\mathscr{A}_a)$.
We have thus constructed a holomorphic extension $\widehat{f}$ of $f$ to the domain $\widehat{Y}:=Y\cup\bigcup_{-\frac\mu2<a<\frac\mu2-2\pi}\bigtriangleup_a\times\mathscr{A}_a\times\mathscr{A}_a$.

Now consider a sequence $\{f_n\}_{n\in\bbN}\subset A^2(\cW_\mu')$ of functions that admit holomorphic extensions to domains $Y_n\supset\overline{\cW_\mu'}$, whence holomorphic extensions $\widehat{f}_n$ to the domains $\widehat{Y}_n:=Y_n\cup\bigcup_{-\frac\mu2<a<\frac\mu2-2\pi}\bigtriangleup_a\times\mathscr{A}_a\times\mathscr{A}_a$. Let us assume $f_n\to f$ in $A^2(\cW_\mu')$, whence $f_n\to f$ uniformly in every compact subset of $\cW_\mu'$. Taking into account, for all $a\in\left(-\frac\mu2,\frac\mu2-2\pi\right)$ and for all $(z_1,z_2,z_3)\in\bigtriangleup_a\times\mathscr{A}_a\times\mathscr{A}_a$, the Cauchy integral formula
\begin{align*}
\widehat{f}_n(z_1,z_2,z_3) &= \frac{1}{(2\pi i)^2} \int_{\partial\mathscr{A}_a}\int_{\partial\mathscr{A}_a} \frac{\widehat{f}_n(z_1,\zeta_2,\zeta_3)}{(\zeta_2-z_2)(\zeta_3-z_3)}\, d\zeta_2 d\zeta_3\\
&= \frac{1}{(2\pi i)^2} \int_{\partial\mathscr{A}_a}\int_{\partial\mathscr{A}_a} \frac{f_n(z_1,\zeta_2,\zeta_3)}{(\zeta_2-z_2)(\zeta_3-z_3)}\, d\zeta_2 d\zeta_3\,,
\end{align*}
we conclude that the sequence $\{\widehat{f}_n\}_{n\in\bbN}$, too, converges uniformly in every compact subset of $\widehat{\cW}_\mu':=\cW_\mu'\cup\bigcup_{-\frac\mu2<a<\frac\mu2-2\pi}\bigtriangleup_a\times\mathscr{A}_a\times\mathscr{A}_a$. These uniform limits define a holomorphic function $\widehat{f}:\widehat{\cW}_\mu'\to\bbC$, which coincides with $f$ in $\cW_\mu'$ by construction.

Now take any $\Omega\in\mathscr{C}_\mu$. By~\eqref{eq:worminclusions}, every $f\in A^2(\overline{\Omega})$ also belongs to $A^2(\overline{\cW_\mu'})$. Thus, $f$ extends holomorphically to $\widehat{\Omega}:=\Omega\cup\bigcup_{-\frac\mu2<a<\frac\mu2-2\pi}\bigtriangleup_a\times\mathscr{A}_a\times\mathscr{A}_a$.

For any $\kappa\in\bbC\setminus\bbZ$, the function $E_\kappa$ (belonging to both $A^2(\cW_\mu')$ and $A^2(\Omega)$) cannot be holomorphically extended to $\bigtriangleup_a\times\mathscr{A}_a\times\mathscr{A}_a$ for any $a\in\left(-\frac\mu2,\frac\mu2-2\pi\right)$. This proves the proper inclusions $A^2(\overline{\cW_\mu'})\subsetneq A^2(\cW_\mu')$ and $A^2(\overline{\Omega})\subsetneq A^2(\Omega)$.
\end{proof}

Our final result concerns the irregularity of the Bergman projection of every $\Omega\in\mathscr{C}_\mu$. Just as in the case of the Diederich-Forn{\ae}ss worm domains, see~\cite{Barrett-Acta}, the Sobolev space $W^{\index,2}(\cW)$ is not preserved by the Bergman projection for sufficiently large $\index$.

\begin{theorem}\label{non-regularity}
Let $\mu>\frac\pi2$ and set $\nu=\frac{\pi}{2\mu}$. For all $\Omega\in\mathscr{C}_\mu$, the Bergman projection associated to $\Omega$ does not map $W^{\index,2}(\Omega)$ into itself when $\index\geq\nu$.
\end{theorem}

Our proof follows the lines of~\cite{Barrett-Acta} and is postponed to the next section.\\\\


\section{Proof of Theorem~\ref{non-regularity}}\label{sec:proof}

Although the proof of Theorem~\ref{non-regularity} closely follows the proof given in~\cite{Barrett-Acta} for the classical worm domain in $\bbC^2$, we prefer to include it for the reader's convenience and to highlight a few novelties due to the increased dimension. We will need several tools, starting with the next remark.

\begin{remark}{\em 
Let $\Omega$ be a domain in $\bbC^n$, invariant with respect to the rotation
\[w\mapsto(w_1,\ldots,w_{m-1},e^{i\theta}w_m,w_{m+1},\ldots,w_n)\]
for all $\theta\in\bbR$. The Bergman
space $A^2(\Omega)$ decomposes as  $\bigoplus_{k \in \bbZ}\cH^{(0,\ldots,k,\ldots,0)}(\Omega)$
where
\begin{align*}
\cH^{(0,\ldots,k,\ldots,0)} (\Omega)&= \left\{F\in A^2(\Omega) : \ F(w_1,\ldots,e^{i\theta}w_m,\ldots,w_n) = e^{ik\theta} F(w)\ \forall\,\theta\in\bbR\right\}\\ 
&= \left\{F\in A^2(\Omega) : F(w)w_m^{-k} \mathrm{\ is\ locally\ constant\ in\ }w_m \right\} \, .
\end{align*}
The projection $Q_{(0,\ldots,k,\ldots,0)}: A^2(\Omega)\to\cH^{(0,\ldots,k,\ldots,0)}(\Omega)$ is given by
\[Q_{(0,\ldots,k,\ldots,0)} F (w) = \frac{1}{2\pi} \int_0^{2\pi} F(w_1,\ldots,e^{i\theta}w_m,\ldots,w_n)e^{-ik\theta}\,d\theta \, .\]
If $\Omega$ is also invariant with respect to rotations in the variable $w_l$, we set
\[\cH^{(0,\ldots,j,\ldots,k,\ldots,0)}:=\cH^{(0,\ldots,j,\ldots,0,\ldots,0)}\cap\cH^{(0,\ldots,0,\ldots,k,\ldots,0)}\,.\]
The projection $Q_{(0,\ldots,j,\ldots,k,\ldots,0)} : A^2(\Omega)\to \cH^{(0,\ldots,j,\ldots,k,\ldots,0)}(\Omega)$ is given by
\begin{align*}
&Q_{(0,\ldots,j,\ldots,k,\ldots,0)} F(w) = Q_{(0,\ldots,j,\ldots,0,\ldots,0)} Q_{(0,\ldots,0,\ldots,k,\ldots,0)} F\\
&= \frac{1}{(2\pi)^2} \int_0^{2\pi} \int_0^{2\pi} F(w_1,\ldots,e^{i\theta_l}w_l,\ldots,e^{i\theta_m}w_m,\ldots,w_n)\,e^{-ik\theta_m}\,d\theta_m\,e^{-ij\theta_l}\,d\theta_l \, .
\end{align*}
}
\end{remark}

Just as in~\cite{Barrett-Acta}, rather than working directly with the domains in the class $\mathscr{C}_\mu$, we construct model domains with the next definition. Here, and in the sequel, we adopt the notations
\begin{align*}
I_{\alpha}&=(-\alpha,\alpha)\, ,\\
S_\beta &= \{\zeta\in\bbC: \Im\zeta\in I_\beta\}
\end{align*}
for all $\alpha,\beta>0$.

\begin{definition} {\em
Fix a real number $\mu > 0$. We set
\begin{align*}
D_\mu &= \left\{ (z_1, z_2, z_3) \in \bbC^3: \Re\left(z _1e^{-i \log |z_2 z_3|^2}\right)>0, \log |z_2|^2,\log |z_3|^2\in I_\mu\right\} \, ,\\
D_\mu' &= \left\{ (w_1, w_2, w_3) \in \bbC^3: \Im(w_1)-\log |w_2 w_3|^2\in I_{\frac\pi2}, \log |w_2|^2,\log |w_3|^2\in I_\mu\right\} \, .
\end{align*}
}
\end{definition}

Clearly, $\cW_\mu'\subset D_\mu$ for all $\mu>0$.

\begin{remark}
If $\Omega\in\mathscr{C}_\mu$, if $\Omega=D_\mu$ or if $\Omega=D_\mu'$, then $\Omega$ is invariant with respect to rotations in the second and third variables. Therefore,
\[A^2(\Omega)=\bigoplus_{j,k\in\bbZ}\cH^{(0,j,k)}(\Omega)\,.\]
For the sake of simplicity, we will denote $\cH^{(0,j,k)}(\Omega)$ as $\cH^{j,k}(\Omega)$.
\end{remark}

More useful tools are provided by the next remarks and lemmas.

\begin{remark}\label{rmk:unwinding} {\em
The domain $D_\mu$ is biholomorphic to $D_\mu'$ via the ``unwinding'' map $z=(z_1, z_2, z_3)\mapsto(\ell(z),z_2, z_3)$, whose inverse is $(w_1, w_2, w_3)\mapsto(e^{w_1}, w_2, w_3)$. Thus, we have an isometric isomorphism $T:A^2(D_\mu)\to A^2(D_\mu')$ with $(TF)(w_1, w_2, w_3)=e^{w_1}F(e^{w_1}, w_2, w_3)$ and $(T^{-1}G)(z_1, z_2, z_3)=z_1^{-1}G(\ell(z),z_2, z_3)$. Clearly, $T$ maps $\cH^{j,k}(D_\mu)$ isomorphically and isometrically onto $\cH^{j,k}(D_\mu')$ for all $j,k\in\bbZ$.
}
\end{remark}

The advantage of $D_\mu'$ over $D_\mu$ is that its fiber over each couple $(w_1,w_2)$ and its fiber over each couple $(w_1,w_3)$ are both connected, a property that allows to prove the next lemma. Before its statement and proof, we recall a few useful properties.

\begin{remark}\label{rmk:paley}
Let $\beta>0$ and consider, on the strip $S_\beta$, a weight $\omega:S_\beta\to\bbR$ of the form $\omega(x+iy)=\alpha(y)$ for some integrable $\alpha:\bbR\to[0,+\infty)$ with $\mathrm{\mathop{supp}}(\alpha)\subseteq I_\beta$. Set $\widetilde \alpha(\xi)=\widehat \alpha(-2i\xi)=\int_{I_\beta}\alpha(y)\,e^{-2y\xi}\,dy$. The Fourier transform associating to any $\phi\in L^2(\bbR,\widetilde \alpha)$ the complex function
\[S_\beta\to\bbC\quad\zeta\mapsto\frac1{2\pi}\int_\bbR \phi(\xi)\,e^{i\zeta\xi}\,d\xi\,,\]
is an isometric isomorphism from $L^2(\bbR,\widetilde \alpha)$ to $A^2(S_\beta,\omega)$ (see~\cite[Proof of Lemma 2]{Barrett-Acta}). In particular, according to~\cite[\S2]{Barrett-Acta}, the reproducing kernel $K_\omega$ of $A^2(S_\beta,\omega)$ has the form
\[K_{\omega}(\zeta,\zeta')=\frac1{2\pi}\int_\bbR \frac{e^{i(\zeta-\bar\zeta')\xi}}{\widehat \alpha(-2i\xi)}\,d\xi\,.\]
\end{remark}

\begin{lemma}\label{lem:unwound}
Let $j,k\in\bbZ,\mu>0$, set $\beta=2\mu+\frac\pi2$ and define $\omega_{j,k}=\omega_{j,k}^\mu:S_\beta\to\bbR$ as
\[\omega_{j,k}(x+iy)=\pi^2(e^{(j+1)(\cdot)}\chi_{I_\mu})*(e^{(k+1)(\cdot)}\chi_{I_\mu})*\chi_{I_{\frac\pi2}}(y)\,.\]
Then
\[L_{j,k}:\cH^{j,k}(D_\mu')\to A^2(S_\beta,\omega_{j,k})\qquad (L_{j,k}G)(\zeta)=G(\zeta, w_2, w_3)w_2^{-j}w_3^{-k}\]
is an isometric isomorphism. Thus: if $K_{\omega_{j,k}}$ denotes the reproducing kernel of $A^2(S_\beta,\omega_{j,k})$, then the reproducing kernel of $\cH^{j,k}(D_\mu')$ is $(w,w')\mapsto K_{\omega_{j,k}}(w_1,w_1')\,(w_2\,\overline w_2')^j\,(w_3\,\overline w_3')^k$. Moreover,
\begin{align*}
K_{\omega_{-1,-1}}(\zeta,\zeta')&=\int_\bbR\frac{\xi^3\,e^{i(\zeta-\bar\zeta')\xi}}{2\pi^3\sinh^2(2\mu\xi)\sinh(\pi\xi)}\,d\xi\,.
\end{align*}
For $\mu>\frac\pi2$, we derive
\begin{align*}
K_{\omega_{-1,-1}}(\zeta,\zeta')&=e^{-(\zeta-\bar\zeta')\nu}(C_\nu(\zeta-\bar\zeta')+C_\nu')+O\left(e^{-\Re(\zeta-\bar\zeta')\nu'}\right)&\mathrm{if\,}\Re(\zeta-\bar\zeta')>0\\
&=e^{(\zeta-\bar\zeta')\nu}(C_\nu(\zeta-\bar\zeta')-C_\nu')+O\left(e^{\Re(\zeta-\bar\zeta')\nu'}\right)&\mathrm{if\,}\Re(\zeta-\bar\zeta')<0\,,
\end{align*}
where $\nu=\frac{\pi}{2\mu}<1, \nu'=\min\{2\nu,1\}>\nu, C_\nu=-\frac{i\nu^5}{2\pi^5\sin(\nu\pi)}$ and $C_\nu'=\frac{i\nu^4}{2\pi^5}\frac{3-\nu\pi\cot(\nu\pi)}{\sin(\nu\pi)}$.
\end{lemma}

\begin{proof}
Let $\pi_1:D_\mu'\to\bbC$ denote the projection map onto the first variable: we have
\begin{align*}
\pi_1(D_\mu')&= \left\{ w_1 \in \bbC: \exists\ t_2,t_3 \in I_\mu \mathrm{\ s.t.\ } \Im(w_1)-t_2-t_3\in I_{\frac\pi2} \right\}\\
& = S_{\beta}\quad \mathrm{with\ }\beta=2\mu+\frac\pi2\,.
\end{align*}
Over each point $x+iy\in S_\beta$, we have the fiber
\begin{align*}
\pi_1^{-1}(x+iy)&=\left\{ (x+iy,w_2, w_3) : \log |w_2|^2,\log |w_3|^2\in I_\mu, y-\log |w_2|^2-\log |w_3|^2\in I_{\frac\pi2} \right\}\,.
\end{align*}
We now construct weights $\omega_{j,k}:S_\beta\to\bbR$ such that
\[M_{j,k}:A^2(S_\beta,\omega_{j,k})\to\cH^{j,k}(D_\mu')\qquad(M_{j,k}f)(w_1, w_2, w_3)=f(w_1)w_2^jw_3^k\]
is an injective linear map and an isometry. For any holomorphic $f,g:S_\beta\to\bbC$, we have
\begin{align*}
\langle M_{j,k}f,M_{j,k}g\rangle_{\cH^{j,k}(D_\mu')}&=\int_{S_\beta}f(w_1)\overline{g(w_1)}\int_{\pi_1^{-1}(w_1)}|w_2|^{2j}|w_3|^{2k}\,dV(w_3)\,dV(w_2)\,dV(w_1)\\
&=\int_{S_\beta}f(w_1)\overline{g(w_1)}\,\omega_{j,k}(w_1)\,dV(w_1)=\langle f,g\rangle_{A^2(S_\beta,\omega_{j,k})}
\end{align*}
if we set, for all $x+iy\in S_\beta$,
\begin{align*}
\omega_{j,k}(x+iy)&=\int_{\pi_1^{-1}(x+iy)}|w_2|^{2j}|w_3|^{2k}\,dV(w_3)\,dV(w_2)\\
&=4\pi^2\iint_{\bbR^2}\chi_{I_\mu}\left(\log r_2^2\right)\,\chi_{I_\mu}\left(\log r_3^2\right)\,\chi_{I_{\frac\pi2}}\left(y-\log r_2^2-\log r_3^2\right)\,r_2^{2j+1}r_3^{2k+1}\,dr_3\,dr_2\\
&=\pi^2\int_\bbR e^{(j+1)s_2}\chi_{I_\mu}(s_2)\int_\bbR e^{(k+1)s_3}\chi_{I_\mu}(s_3)\chi_{I_{\frac\pi2}}(y-s_2-s_3)\,ds_3\,ds_2\\
&=\pi^2\int_\bbR e^{(j+1)s_2}\chi_{I_\mu}(s_2)\left(e^{(k+1)(\cdot)}\chi_{I_\mu}*\chi_{I_{\frac\pi2}}\right)(y-s_2)\,ds_2\\
&=\pi^2\,(e^{(j+1)(\cdot)}\chi_{I_\mu})*(e^{(k+1)(\cdot)}\chi_{I_\mu})*\chi_{I_{\frac\pi2}}(y)\,.
\end{align*}

We now fix any $G\in\cH^{j,k}(D_\mu')\subset\cH^{k}(D_\mu')$, recalling that $G(w_1, w_2, w_3)w_3^{-k}$ is locally constant in $w_3$ and $G(w_1, w_2, w_3)w_2^{-j}$ is locally constant in $w_2$. We are going to show that the function $G(w_1, w_2, w_3)w_2^{-j}w_3^{-k}$ is constant in $(w_2,w_3)$: the resulting function of $w_1$ will automatically be an element of $A^2(S_\beta,\omega_{j,k})$ (which proves that $M_{j,k}$ is surjective and an isomorphism, as desired). Let $\pi_{1,2}:D_\mu'\to\bbC^2$ denote the projection map onto the first two variables: then
\begin{align*}
\pi_{1,2}(D_\mu')&= \left\{ (w_1,w_2) \in \bbC^2: \log |w_2|^2\in I_\mu,\,\exists\ t_3 \in I_\mu \mathrm{\ s.t.\ } \Im(w_1)-\log |w_2|^2-t_3\in I_{\frac\pi2} \right\}\\
&= \left\{ (w_1,w_2) \in \bbC^2: \log |w_2|^2\in I_\mu,\,\Im(w_1)-\log |w_2|^2\in I_{\mu+\frac\pi2} \right\}\,.
\end{align*}
Over each pair $(w_1,w_2)\in\pi_{1,2}(D_\mu')$, we have the fiber
\begin{align*}
&\pi_{1,2}^{-1}(w_1,w_2)\\
&=\left\{ (w_1, w_2, w_3) : \log |w_3|^2\in I_\mu \cap \left(\Im(w_1)-\log |w_2|^2-\frac\pi2,\Im(w_1)-\log |w_2|^2+\frac\pi2\right)\right\}\,,
\end{align*}
which is a nonempty annulus because $\Im(w_1)-\log |w_2|^2-\frac\pi2<\mu$ and $\Im(w_1)-\log |w_2|^2+\frac\pi2>-\mu$. Since every fiber $\pi_{1,2}^{-1}(w_1,w_2)$ is connected, the function $G(w_1, w_2, w_3)w_3^{-k}$ is not only locally constant but constant in $w_3$. An analogous argument proves that the function $G(w_1, w_2, w_3)w_2^{-j}$ is not only locally constant but constant in $w_2$. The proof of the first statement is complete.

The fact that $M_{j,k}:A^2(S_\beta,\omega_{j,k})\to\cH^{j,k}(D_\mu')$, which maps $f$ into $(w_1,w_2,w_3)\mapsto f(w_1)w_2^jw_3^k$, is an isometry yields the second statement.

Our next aim is computing and estimating the reproducing kernel $K_{\omega_{-1,-1}}$ of $A^2(S_\beta,\omega_{-1,-1})$. We have
\[\omega_{-1,-1}(x+iy)=\pi^2\,\chi_{I_\mu}*\chi_{I_\mu}*\chi_{I_{\frac\pi2}}(y)\]
and the Fourier-Laplace transform of the latter function, computed at $-2i\xi$, is
\[\pi^2\ (\widehat\chi_{I_\mu})^2(-2i\xi)\ \widehat\chi_{I_{\frac\pi2}}(-2i\xi) =\pi^2\xi^{-3}\sinh^2(2\mu\xi)\sinh(\pi\xi)\,.\]
By Remark~\ref{rmk:paley},
\begin{align*}
K_{\omega_{-1,-1}}(\zeta,\zeta')&=\int_\bbR\cI(\xi)\,d\xi\,,\qquad
\cI(\xi)=\frac{\xi^3\,e^{i(\zeta-\bar\zeta')\xi}}{2\pi^3\sinh^2(2\mu\xi)\sinh(\pi\xi)}\,.
\end{align*}
When viewed as a function of a complex variable $\xi=x+iy$, the integrand function $\cI(\xi)$ has squared modulus
\[|\cI(x+iy)|^2=\frac{(x^2+y^2)^3\,e^{-2\Im(\zeta-\bar\zeta')x}\,e^{-2\Re(\zeta-\bar\zeta')y}}{2\pi^3(\sinh^2(2\mu x)+\sin^2(2\mu y))^2(\sinh^2(\pi x)+\sin^2(\pi y))}\,.\]
Moreover, $\cI$ has: a simple pole at each point $ik$ with $k\in\bbZ^*$, because
\[\sinh(\pi\xi)=(-1)^k\sum_{n\in\bbN}\frac{(\pi\xi-ik\pi)^{2n+1}}{(2n+1)!}=(-1)^k\sum_{n\in\bbN}\frac{\pi^{2n+1}}{(2n+1)!}(\xi-ik)^{2n+1}\,;\]
a double pole at each point $ik\nu$ with $k\in\bbZ^*$, where $\nu=\frac{\pi}{2\mu}$, because
\begin{align*}
\sinh(2\mu\xi)&
=(-1)^k\sum_{n\in\bbN}\frac{(2\mu)^{2n+1}}{(2n+1)!}(\xi-ik\nu)^{2n+1}=(-1)^k\left(2\mu(\xi-ik\nu)+\frac{4}{3}\mu^3(\xi-ik\nu)^3+\ldots\right)\,.
\end{align*}
Our hypothesis $\mu>\frac\pi2$ yields $\nu<1$. Thus, the poles nearest to $\bbR$ are $\pm i\nu$, followed by $\pm i\nu'$ with $\nu'=\min\{2\nu,1\}>\nu$. If $\Re(\zeta-\bar\zeta')>0$, let $y_0\in(\nu,\nu')$ and $p_0=i\nu$; if $\Re(\zeta-\bar\zeta')<0$, let $y_0\in(-\nu',-\nu)$ and $p_0=-i\nu$. Along the line $y=y_0$, the modulus $|\cI(x+iy)|$ decays well as $x\to\pm\infty$. A standard contour integration argument yields that
\begin{equation}\label{eq:contourintegration}
K_{\omega_{-1,-1}}(\zeta,\zeta')=\int_\bbR \cI(\xi)\,d\xi=\int_{\bbR+iy_0} \cI(\xi)\,d\xi+\mathop{\mathrm{Res}}(\cI,p_0)\,.
\end{equation}
Clearly,
\[\int_{\bbR+iy_0} \cI(\xi)\,d\xi=O\left(e^{-\Re(\zeta-\bar\zeta')y_0}\right)\,.\]
To compute $\mathop{\mathrm{Res}}(\cI,p_0)$ at $p_0=\pm i\nu$, we first define the functions $\cG_1^\pm(\xi)=(\xi\mp i\nu)\sinh^{-1}(2\mu\xi)$, which have $\cG_1^\pm(\pm i\nu)=-(2\mu)^{-1}=-\frac\nu\pi$ and $(\cG_1^\pm)'(\pm i\nu)=0$, as well as the function $\cG_2(\xi)=(2\pi^3)^{-1}\xi^3\sinh^{-1}(\pi\xi)e^{i(\zeta-\bar\zeta')\xi}$. We then compute
\begin{align*}
\mathop{\mathrm{Res}}(\cI,\pm i\nu)&=\lim_{\xi\to\pm i\nu}\frac{d}{d\xi}\left((\xi\mp i\nu)^2\cI(\xi)\right)
=\lim_{\xi\to\pm i\nu}\frac{d}{d\xi}\left(\cG_1^\pm(\xi)^2\cG_2(\xi)\right)\\
&=2\cG_1^\pm(\pm i\nu)\,(\cG_1^\pm)'(\pm i\nu)\,\cG_2(\pm i\nu)+\cG_1^\pm(\pm i\nu)^2\,\cG_2'(\pm i\nu)\\
&=0+\frac{	\nu^2}{\pi^2}\cG_2'(\pm i\nu)\\
&=\frac{\nu^2}{\pi^2}\left(\frac{\xi^2e^{i(\zeta-\bar\zeta')\xi}}{2\pi^3\sinh(\pi\xi)}\left(3-\pi\xi\coth(\pi\xi)+i(\zeta-\bar\zeta')\xi\right)\right)_{|_{\xi=\pm i\nu}}\\
&=e^{\mp(\zeta-\bar\zeta')\nu}(C_\nu(\zeta-\bar\zeta')\pm C_\nu')\,,
\end{align*}
where
\begin{align*}
C_\nu&=\frac{\nu^2}{2\pi^5}\frac{i\xi^3}{\sinh(\pi\xi)}_{|_{\xi=\pm i\nu}}=\frac{\nu^5}{2\pi^5}\frac{1}{\sinh(i\nu\pi)}=\frac{\nu^5}{2\pi^5}\frac{1}{i\sin(\nu\pi)}=-\frac{i\nu^5}{2\pi^5\sin(\nu\pi)}\\
C_\nu'&=-\frac{\nu^4}{2\pi^5}\frac{3- i\nu\pi\coth(i\nu\pi)}{\sinh(i\nu\pi)}=\frac{\nu^4}{2\pi^5}\frac{3-\nu\pi\cot(\nu\pi)}{-i\sin(\nu\pi)}=\frac{i\nu^4}{2\pi^5}\frac{3-\nu\pi\cot(\nu\pi)}{\sin(\nu\pi)}\,.
\end{align*}
If $\Re(\zeta-\bar\zeta')>0$, by letting $y_0\to(\nu')^-$ in formula~\eqref{eq:contourintegration}, we find that
\[K_{\omega_{-1,-1}}(\zeta,\zeta')=e^{-(\zeta-\bar\zeta')\nu}(C_\nu(\zeta-\bar\zeta')+C_\nu')+O\left(e^{-\Re(\zeta-\bar\zeta')\nu'}\right)\,.\]
Similarly, if $\Re(\zeta-\bar\zeta')<0$, by letting $y_0\to(-\nu')^+$ in formula~\eqref{eq:contourintegration}, we find that
\[K_{\omega_{-1,-1}}(\zeta,\zeta')=e^{(\zeta-\bar\zeta')\nu}(C_\nu(\zeta-\bar\zeta')-C_\nu')+O\left(e^{\Re(\zeta-\bar\zeta')\nu'}\right)\,,\]
as desired.
\end{proof}

The nature of the asymptotic expansion of $K_{\omega_{-1,-1}}(\zeta,\zeta')$ in Lemma~\ref{lem:unwound} is a slight novelty with respect to~\cite[\S2]{Barrett-Acta}, due to the increased dimension of our worm domains. This novelty is related to the presence of a double pole of $\cI(\xi)$ at $\xi=\pm i\nu$, in contrast with the simple pole of the integrand in~\cite[Equation (2.1)]{Barrett-Acta}.

For future use, we prove the next lemma.

\begin{lemma}\label{lem:increasedradius}
Let $0<\mu<\mu'$. Associating to each $F\in A^2(D_{\mu'})$ its restriction $F_{|_{D_\mu}}$ defines a linear map
\[A^2(D_{\mu'})\to A^2(D_\mu)\,,\]
whose image is dense in $A^2(D_\mu)$.
\end{lemma}

\begin{proof}
Thanks to Remark~\ref{rmk:unwinding} and to Lemma~\ref{lem:unwound}, it suffices to prove that, after setting $\beta=2\mu+\frac\pi2$ and $\beta'=2\mu'+\frac\pi2$, for all $j,k\in\bbZ$ the restriction to $S_\beta$ defines a linear map
\[A^2(S_{\beta'},\omega_{j,k}^{\mu'})\to A^2(S_\beta,\omega_{j,k}^\mu)\,,\]
whose image is dense in $A^2(S_\beta,\omega_{j,k}^\mu)$.

Since
\begin{align*}
\omega_{j,k}^\mu(x+iy)&=\pi^2(e^{(j+1)(\cdot)}\chi_{I_\mu})*(e^{(k+1)(\cdot)}\chi_{I_\mu})*\chi_{I_{\frac\pi2}}(y)\,,\\
\omega_{j,k}^{\mu'}(x+iy)&=\pi^2(e^{(j+1)(\cdot)}\chi_{I_{\mu'}})*(e^{(k+1)(\cdot)}\chi_{I_{\mu'}})*\chi_{I_{\frac\pi2}}(y)\,,
\end{align*}
our Remark~\ref{rmk:paley} guarantees that the spaces $A^2(S_\beta,\omega_{j,k}^\mu),A^2(S_{\beta'},\omega_{j,k}^{\mu'})$ are the images through the inverse Fourier transform of the spaces $L^2(\bbR,\phi),L^2(\bbR,\psi)$, where
\begin{align*}
\phi(\xi)&=\pi^2\,\frac{\sinh\left(2\mu\left(\xi-\frac{j+1}2\right)\right)}{\xi-\frac{j+1}2}\,\frac{\sinh\left(2\mu\left(\xi-\frac{k+1}2\right)\right)}{\xi-\frac{k+1}2}\,\frac{\sinh(\pi\xi)}{\xi}\,,\\
\psi(\xi)&=\pi^2\,\frac{\sinh\left(2\mu'\left(\xi-\frac{j+1}2\right)\right)}{\xi-\frac{j+1}2}\,\frac{\sinh\left(2\mu'\left(\xi-\frac{k+1}2\right)\right)}{\xi-\frac{k+1}2}\,\frac{\sinh(\pi\xi)}{\xi}\,.
\end{align*}
It follows at once that the restriction to $S_\beta$ maps $A^2(S_{\beta'},\omega_{j,k}^{\mu'})$ into $A^2(S_\beta,\omega_{j,k}^\mu)$. Since $C^\infty_0(\bbR)$ is dense in both $L^2(\bbR,\phi)$ and $L^2(\bbR,\psi)$, we immediately conclude that the image of the restriction map $A^2(S_{\beta'},\omega_{j,k}^{\mu'})\to A^2(S_\beta,\omega_{j,k}^\mu)$ is dense in $A^2(S_\beta,\omega_{j,k}^\mu)$.
\end{proof}

We now turn back to $D_\mu$.

\begin{lemma}\label{lem:l2convergence}
The reproducing kernel of $\cH^{j,k}(D_\mu)$ is
\[K_{j,k}(z,z')=K_{\omega_{j,k}}(\ell(z),\ell(z'))(z_1\,\overline z_1')^{-1}(z_2\,\overline z_2')^j\,(z_3\,\overline z_3')^k\,.\]
In particular, recalling that (for $\kappa\in\bbC$) $E_\kappa(z)$ is the holomorphic extension of $z_1^\kappa$ constructed in Definition~\ref{def:logarithm}, we have
\begin{align}\label{eq:wound}
&z_2\,\overline z_2'\,z_3\,\overline z_3'\,K_{-1,-1}(z,z')=\int_\bbR \frac{\xi^3\,E_{i\xi-1}(z)\,\overline{E_{i\xi-1}(z')}
}{2\pi^3\sinh^2(2\mu\xi)\sinh(\pi\xi)}\,d\xi\,.
\end{align}
For $\mu>\frac\pi2, \nu=\frac{\pi}{2\mu}, \nu'=\min\{2\nu,1\}, C_\nu=-\frac{i\nu^5}{2\pi^5\sin(\nu\pi)}$ and $C_\nu'=\frac{i\nu^4}{2\pi^5}\frac{3-\nu\pi\cot(\nu\pi)}{\sin(\nu\pi)}$, we have
\begin{align*}
&z_2\,\overline z_2'\,z_3\,\overline z_3'\,K_{-1,-1}(z,z')\\
&=E_{-\nu-1}(z)\,\overline{E_{\nu-1}(z')}\left(C_\nu\left(\ell(z)-\overline{\ell(z')}\right)+C_\nu'\right)+O\left(\frac{|z_1'|^{\nu'-1}}{|z_1|^{\nu'+1}}\right)\quad\mathrm{if\,}|z_1|>|z_1'|\notag\\
&=E_{\nu-1}(z)\,\overline{E_{-\nu-1}(z')}\left(C_\nu\left(\ell(z)-\overline{\ell(z')}\right)-C_\nu'\right)+O\left(\frac{|z_1|^{\nu'-1}}{|z_1'|^{\nu'+1}}\right)\quad\mathrm{if\,}|z_1|<|z_1'|.\notag
\end{align*}
As a consequence: for each $\mu>\frac\pi2,z'\in D_\mu,m\in\bbN,s\in[0,1)$, the function
\[z\mapsto \left(\Re\left(z_1e^{-i \log |z_2 z_3|^2}\right)\right)^s\,\frac{\partial^m}{\partial z_1^m}K_{-1,-1}(z,z')\]
does not belong to $L^2(D_\mu)$ if $m-s\geq\nu$.
\end{lemma}

\begin{proof}
The first statement follows from Lemma~\ref{lem:unwound} through the isometry $T^{-1}:\cH^{j,k}(D_\mu')\to\cH^{j,k}(D_\mu)$, which maps each $F$ into $T^{-1}F(z)=F(\ell(z),z_2,z_3)z_1^{-1}$.

To prove formula~\eqref{eq:wound}, we compute
\begin{align*}
z_2\,\overline z_2'\,z_3\,\overline z_3'\,K_{-1,-1}(z,z')&=(z_1\,\overline z_1')^{-1}K_{\omega_{-1,-1}}(\ell(z),\ell(z'))\\
&=(z_1\,\overline z_1')^{-1}\int_\bbR \frac{\xi^3\,e^{i(\ell(z)-\overline{\ell(z')})\xi}}{2\pi^3\sinh^2(2\mu\xi)\sinh(\pi\xi)}\,d\xi\\
&=\int_\bbR \frac{\xi^3\,E_{i\xi-1}(z)\,\overline{E_{i\xi-1}(z')}
}{2\pi^3\sinh^2(2\mu\xi)\sinh(\pi\xi)}\,d\xi\,,
\end{align*}
where we took into account the equalities $E_{i\xi-1}(z)=e^{(i\xi-1)\ell(z)}=e^{i\ell(z)\xi}z_1^{-1}$ and the equalities $\overline{E_{i\xi-1}(z')}=\overline{e^{i\ell(z')\xi}(z'_1)^{-1}}=e^{-i\overline{\ell(z')}\xi}(\bar z'_1)^{-1}$.

We assume henceforth $\mu>\frac\pi2$ and apply the estimates for $K_{\omega_{-1,-1}}$ obtained in Lemma~\ref{lem:unwound}. For $|z_1|>|z_1'|$, which is the same as $\Re\left(\ell(z)-\overline{\ell(z')}\right)>0$, we find that
\begin{align*}
z_1\,\overline z_1'\,z_2\,\overline z_2'\,z_3\,\overline z_3'\,K_{-1,-1}(z,z')&=e^{-(\ell(z)-\overline{\ell(z')})\nu}\left(C_\nu\left(\ell(z)-\overline{\ell(z')}\right)+C_\nu'\right)+O\left(e^{-\Re\left(\ell(z)-\overline{\ell(z')}\right)\nu'}\right)\\
&=E_{-\nu}(z)\,\overline{E_{\nu}(z')}\left(C_\nu\left(\ell(z)-\overline{\ell(z')}\right)+C_\nu'\right)+O\left(\left|z_1'/z_1\right|^{\nu'}\right)\,,\\
z_2\,\overline z_2'\,z_3\,\overline z_3'\,K_{-1,-1}(z,z')&=E_{-\nu-1}(z)\,\overline{E_{\nu-1}(z')}\left(C_\nu\left(\ell(z)-\overline{\ell(z')}\right)+C_\nu'\right)+O\left(\frac{|z_1'|^{\nu'-1}}{|z_1|^{\nu'+1}}\right)\,,
\end{align*}
as stated. For $|z_1|<|z_1'|$, we find that
\begin{align*}
z_1\,\overline z_1'\,z_2\,\overline z_2'\,z_3\,\overline z_3'\,K_{-1,-1}(z,z')&=e^{(\ell(z)-\overline{\ell(z')})\nu}\left(C_\nu\left(\ell(z)-\overline{\ell(z')}\right)-C_\nu'\right)+O\left(e^{\Re\left(\ell(z)-\overline{\ell(z')}\right)\nu'}\right)\\
&=E_{\nu}(z)\,\overline{E_{-\nu}(z')}\left(C_\nu\left(\ell(z)-\overline{\ell(z')}\right)-C_\nu'\right)+O\left(\left|z_1/z_1'\right|^{\nu'}\right)\,,\\
z_2\,\overline z_2'\,z_3\,\overline z_3'\,K_{-1,-1}(z,z')&=E_{\nu-1}(z)\,\overline{E_{-\nu-1}(z')}\left(C_\nu\left(\ell(z)-\overline{\ell(z')}\right)-C_\nu'\right)+O\left(\frac{|z_1|^{\nu'-1}}{|z_1'|^{\nu'+1}}\right)\,,
\end{align*}
as desired.

We now turn to the last statement, still under the assumption $\mu>\frac\pi2$. Since $\frac{\partial E_{\kappa}}{\partial z_1}=\kappa E_{\kappa-1}$ and $\frac{\partial}{\partial z_1}\left(E_{\kappa}(z)\,\ell(z)\right)=\kappa E_{\kappa-1}(z)\,\ell(z)-E_{\kappa}(z)z_1^{-1}=E_{\kappa-1}\left(\kappa\,\ell(z)-1\right)$, for any $m\in\bbN$ and any $z'\in D_\mu$ there exist $\phi_{\nu}^m=\phi_{\nu}^m(z'),\psi_{\nu}^m=\psi_{\nu}^m(z')\in\bbC$ such that
\begin{align*}
z_2\,z_3\,\frac{\partial^m}{\partial z_1^m}K_{-1,-1}(z,z')&=E_{\nu-m-1}(z)\left(\phi_{\nu}^m\ell(z)+\psi_{\nu}^m\right)+O\left(|z_1|^{\nu'-m-1}\right)
\end{align*}
in the region $D_\mu^{z'}:=\{z\in D_\mu : |z_1|<|z_1'|\}$. If the function $\frac{\partial^mK_{-1,-1}}{\partial z_1^m}(\cdot,z')$ belongs to $L^2(D_\mu)$, then $E_{\nu-m-1}(z)\left(\phi_{\nu}^m\ell(z)+\psi_{\nu}^m\right)\left(z_2\,z_3\right)^{-1}$ is square-integrable in $D_\mu^{z'}$. Using the notation $H^{z'}$ for the half-disk $\{\zeta\in\bbC : \Re(\zeta)>0, |\zeta|<|z_1'|\}$, we compute
\begin{align*}
&\int_{D_\mu^{z'}}|E_{\nu-m-1}(z)|^2\left|\phi_{\nu}^m\ell(z)+\psi_{\nu}^m\right|^2\left|z_2\,z_3\right|^{-2}dV(z)\\
&=4\pi^2\int_{\log r_3^2\in I_\mu}\int_{\log r_2^2\in I_\mu}\int_{z _1e^{-i \log(r_2r_3)^2}\in H^{z'}}|z_1|^{2(\nu-m-1)}\left|\phi_{\nu}^m\,I+\psi_{\nu}^m\right|^2dV(z_1)\frac{dr_2}{r_2}\frac{dr_3}{r_3}\\
&=\pi^2\int_{s_3\in I_\mu}\int_{s_2\in I_\mu}\int_{z _1e^{-i (s_2+s_3)}\in H^{z'}}|z_1|^{2(\nu-m-1)}\left|\phi_{\nu}^m\,I\!I+\psi_{\nu}^m\right|^2dV(z_1)\,ds_2\,ds_3\\
&=\pi^2\int_{s_3\in I_\mu}\int_{s_2\in I_\mu}\int_{s_2+s_3-\frac\pi2}^{s_2+s_3+\frac\pi2}\int_0^{|z_1'|}r_1^{2(\nu-m)-1}\left|\phi_{\nu}^m\,I\!I\!I+\psi_{\nu}^m\right|^2\,dr_1\,d\theta_1\,ds_2\,ds_3\,,
\end{align*}
where
\begin{align*}
I&=\log\left(z_1e^{-i \log(r_2r_3)^2}\right)+i \log(r_2r_3)^2\\
I\!I&=\log\left(z_1e^{-i (s_2+s_3)}\right)+i (s_2+s_3)\\
I\!I\!I&=\log\left(r_1e^{i (\theta_1-s_2-s_3)}\right)+i (s_2+s_3)=\log r_1+i\theta_1\,.
\end{align*}
If the last integral converges, then $\nu-m>0$. Similarly, if $s\in[0,1),m\in\bbN,z'\in D_\mu$ and if the function
\[z\mapsto \left(\Re\left(z_1e^{-i \log |z_2 z_3|^2}\right)\right)^s\,\frac{\partial^m}{\partial z_1^m}K_{-1,-1}(z,z')\]
belongs to $L^2(D_\mu)$, then $\nu+s-m>0$, i.e., $m-s<\nu$.
\end{proof}

The slight novelties we already mentioned (with respect to~\cite[\S2]{Barrett-Acta}) in the asymptotic expansion in Lemma~\ref{lem:unwound} are reflected in our asymptotic expansion of $K_{-1,-1}(z,z')$. Nevertheless, the condition $m-s\geq\nu$ appearing at the end of Lemma~\ref{lem:l2convergence} is the same exact condition appearing in~\cite[Equation (3.2)]{Barrett-Acta}.  As a consequence: the condition $\index\geq\nu$ at the end of Theorem~\ref{non-regularity} about the irregularity of the Bergman projection of every $\Omega\in\mathscr{C}_\mu$, which we are about to prove, is the same exact condition appearing in~\cite[Theorem 1]{Barrett-Acta}. For the sake of simplicity, we will write $W^\index(\Omega)$ instead of $W^{\index,2}(\Omega)$.

\begin{proof}[Proof of Theorem~\ref{non-regularity}]
Fix $\Omega=\cW_\eta\in\mathscr{C_\mu}$. For $\lambda\geq1$, we set $\tau_\lambda(z)=(\lambda z_1,z_2,z_3)$ for all $z\in\Omega$ and $\Omega^\lambda=\tau_\lambda(\Omega)$. Setting $T_\lambda f = f\circ\tau_\lambda$ defines an isomorphism $T_\lambda:W^\index(\Omega^\lambda)\to W^\index(\Omega)$ for all $\index\geq0$, with inverse $T_\lambda^{-1}g=g\circ\tau_\lambda^{-1}$. Although $T_\lambda$ is not an isometry in general, we have
\[\Vert T_\lambda f\Vert_{W^\index(\Omega)}\leq\lambda^{\index-1}\Vert f\Vert_{W^\index(\Omega^\lambda)}\]
for all $\index\in\bbN$ and, by interpolation, for all $\index\geq0$. Consider the defining function
\[\rho(z)=|z_1|^2-2\Re(z_1e^{-i\log|z_2z_3|^2})+\eta(\log|z_2|^2,\log|z_3|^2)\]
of $\Omega=\cW_\eta$ and notice that
\[\rho_\lambda(z)=\lambda\rho\circ\tau_\lambda^{-1}(z)=\lambda^{-1}|z_1|^2-2\Re(z_1e^{-i\log|z_2z_3|^2})+\lambda\eta(\log|z_2|^2,\log|z_3|^2)\]
is a defining function of $\Omega^\lambda$ with
\[\rho_\lambda(z)\,\chi_{I_\mu}(\log|z_2|^2)\,\chi_{I_\mu}(\log|z_3|^2)=\left(\lambda^{-1}|z_1|^2-2\Re(z_1e^{-i\log|z_2z_3|^2})\right)\,\chi_{I_\mu}(\log|z_2|^2)\,\chi_{I_\mu}(\log|z_3|^2)\]
converging as $\lambda\to+\infty$ to the defining function of $D_\mu$ given by
\[\rho_\infty(z)\,\chi_{I_\mu}(\log|z_2|^2)\,\chi_{I_\mu}(\log|z_3|^2)\,,\qquad\rho_\infty(z)=-2\Re(z_1e^{-i\log|z_2z_3|^2})\,.\]
We note, for future reference, that $(\Omega^\lambda\cap D_\mu)\nearrow D_\mu$ as $\lambda\nearrow+\infty$, whence every compact subset of $D_\mu$ is contained in $\Omega^\lambda$ for sufficiently large $\lambda$; while every compact subset of $\bbC^3\setminus\overline{D}_\mu$ does not intersect $\Omega^\lambda$ for sufficiently large $\lambda$. Let $P,P_\lambda$ denote the Bergman projections of $\Omega,\Omega^\lambda$, respectively, related by the equality $P_\lambda=T_\lambda^{-1}P T_\lambda$. Then
\begin{align*}
\left\Vert|\rho_\lambda|^s\,\frac{\partial^m}{\partial z_1^m}P_\lambda f\right\Vert_{L^2(\Omega^\lambda)}
&=\left\Vert|\rho_\lambda|^s\,\frac{\partial^m}{\partial z_1^m}(T_\lambda^{-1}P T_\lambda f)\right\Vert_{L^2(\Omega^\lambda)}\\
&=\lambda^{s-m}\left\Vert\left|\rho\circ\tau_\lambda^{-1}\right|^s\,\left(\frac{\partial^m}{\partial z_1^m}(P T_\lambda f)\right)\circ\tau_\lambda^{-1}\right\Vert_{L^2(\Omega^\lambda)}\\
&=\lambda^{1-m+s}\left\Vert|\rho|^s\,\frac{\partial^m}{\partial z_1^m}(P T_\lambda f)\right\Vert_{L^2(\Omega)}\\
&\leq C_1\,\lambda^{1-m+s}\left\Vert P T_\lambda f\right\Vert_{W^{m-s}(\Omega)}
\end{align*}
for some constant $C_1>0$, according to~\cite{Li87}. Assume that $P(W^{m-s}(\Omega))\subseteq W^{m-s}(\Omega)$, whence 
there exists a constant $C>0$ such that $\left\Vert P g\right\Vert_{W^{m-s}(\Omega)}\leq C\,\left\Vert g\right\Vert_{W^{m-s}(\Omega)}$. If we set $C_2:=C_1\,C$, then the previous chain of inequalities yields
\begin{align}
\left\Vert|\rho_\lambda|^s\,\frac{\partial^m}{\partial z_1^m}P_\lambda f\right\Vert_{L^2(\Omega^\lambda)}
&\leq C_2\,\lambda^{1-m+s}\left\Vert T_\lambda f\right\Vert_{W^{m-s}(\Omega)}\notag\\
&\leq C_2\,\left\Vert f\right\Vert_{W^{m-s}(\Omega^\lambda)}\,.\label{eq:technical1}
\end{align}
We claim that inequality~\eqref{eq:technical1} implies that
\begin{equation}\label{eq:technical2}
\left\Vert|\rho_\infty|^s\,\frac{\partial^m}{\partial z_1^m}P_\infty f\right\Vert_{L^2(D_\mu)}\leq C_2\,\left\Vert f\right\Vert_{W^{m-s}(\bbC^3)}
\end{equation}
for all $f\in W^{m-s}(\bbC^3)$ compactly supported in $\overline{D}_\mu$. Assuming this claim, let $K(\cdot,\cdot)$ denote the Bergman kernel of $D_\mu$. For any fixed $z'\in D_\mu$ and for any open ball $B=B(z',R)\subset D_\mu$, we can choose a function $f\in C^\infty_0(\bbC^3)$, supported in $B$ and radial in $B$, such that $K(\cdot,z')=P_\infty f$: it suffices, see~\cite{belli}, to set $f(z'+su)=\phi(s)$ when $s>0, u\in\bbC^3, |u|=1$, for some $\phi:[0,+\infty)\to \bbC$ supported in the interval $[0,R]$ and such that $\int_{\bbC^3} f(w) dV(w)=1$. From the last inequality, we conclude that
\[\left\Vert|\rho_\infty|^s\,\frac{\partial^m}{\partial z_1^m}K(\cdot,z')\right\Vert_{L^2(D_\mu)}<+\infty\,,\]
whence $m-s<\nu$ because of Lemma~\ref{lem:l2convergence}.

Our claim can be proven as follows. For sufficiently large $\lambda$, the function $\chi_{\Omega^\lambda}\,P_\lambda f$ is well defined and
\[\Vert\chi_{\Omega^\lambda}\,P_\lambda f\Vert_{L^2(\bbC^3)}\leq\Vert f\Vert_{L^2(\bbC^3)}\,.\]
Therefore, there exists a sequence $\{\lambda_n\}_{n\in\bbN}\subset(1,+\infty)$ with $\lambda_n\to+\infty$ as $n\to+\infty$ such that the sequence $\{\chi_{\Omega^{\lambda_n}}\,P_{\lambda_n} f\}_{n\in\bbN}$ (is well defined and) weakly converges to a function $h$ in $L^2(\bbC^3)$. Since $\{\Omega^{\lambda_n}\cap D_\mu\}_{n\in\bbN}$ is an increasing sequence of domains converging to $D_\mu$, the limit function $h$ must be holomorphic in $D_\mu$. Moreover, $h\equiv0$ in $\bbC^3\setminus\overline{D}_\mu$. Our next aim is showing that $h=\chi_{D_\mu}\,P_\infty f$. If we choose $\mu'>\mu$ such that the inclusions~\eqref{eq:worminclusions} hold true, then for every $n\in\bbN$ the inclusion $\Omega^{\lambda_n}\subset D_{\mu'}$ holds true. Thus, the property $f-P_{\lambda_n}f\perp A^2(\Omega^{\lambda_n})$ implies the property $f-\chi_{\Omega^{\lambda_n}}\,P_{\lambda_n}f\perp A^2(D_{\mu'})$. By taking the weak limit as $n\to+\infty$, we conclude that $f-h\perp A^2(D_{\mu'})$. Lemma~\ref{lem:increasedradius} now yields $f-h\perp A^2(D_\mu)$. Therefore, $h=\chi_{D_\mu}\,P_\infty f$, as desired. Finally, we remark that $\chi_{D_\mu}\,|\rho_\infty|^s\,\frac{\partial^m}{\partial z_1^m}P_\infty f$ is the weak limit in $L^2(\bbC^3)$ of a subsequence of
\[\left\{\chi_{\Omega^{\lambda_n}}\,|\rho_\lambda|^s\,\frac{\partial^m}{\partial z_1^m}P_{\lambda_n} f\right\}_{n\in\bbN}\,.\]
Taking into account inequality~\eqref{eq:technical1}, we immediately derive inequality~\eqref{eq:technical2}. This completes the proof of our claim.
\end{proof}


\section{Closing remarks}\label{sec:closingremarks}

For more than forty-five years the worm domain has provided analytic and geometric insight into important complex analytic phenomena of several variables. The worm has been a decisive counterexample for many longstanding problems.

In this paper we have constructed some new, geometrically natural, 3-dimensional variants of the classical two-dimensional Diederich-Fornæss worm domain. We show that they are smoothly bounded, pseudoconvex, and have nontrivial Nebenh\"{u}lle. We also show that their Bergman projections are not bounded in the Sobolev topology for sufficiently large Sobolev indices.

We plan to further this study, in the spirit of~\cite[\S3]{ADF2023}, to fully understand which geometric features of our newly constructed class of smoothly bounded pseudoconvex domains play an essential role. We expect from this further study a significant generalization step in the already rich realm of worm domains.


\section*{Acknowledgements}

The second author is partly supported by GNAMPA INdAM through Progetto ``Function theory in several complex and quaternionic variables'' (CUP{\_}E53C220011930001).

The third author is partly supported by: GNSAGA INdAM; Progetto ``Teoria delle funzioni ipercomplesse e applicazioni'' Universit\`a di Firenze; PRIN 2022 ``Real and complex manifolds: geometry and holomorphic dynamics'' MIUR; Finanziamento Premiale ``Splines for accUrate NumeRics: adaptIve models for Simulation Environments'' INdAM.

The authors are grateful to the anonymous referee for making valuable suggestions that improved the presentation of this work.


\vspace*{.2in}

\begin{quote} 
Steven G. Krantz \\ 
Department of Mathematics \\
Washington University in St. Louis \\
1 Brookings Drive\\
St.\ Louis, Missouri 63130 \\ 
{\tt sk@wustl.edu} 
\end{quote}
\vspace*{.01in}

\begin{quote}
Marco M. Peloso \\
Dipartimento di Matematica ``F. Enriques''\\
(Dipartimento di Eccellenza MUR 2023-27)\\
Universit\`a degli Studi di Milano\\
Via C. Saldini 50\\
I-20133 Milano \\
{\tt marco.peloso@unimi.it}
\end{quote}
\vspace*{.01in}

\begin{quote}
Caterina Stoppato \\
Dipartimento di Matematica e Informatica``U. Dini''\\
Universit\`a degli Studi di Firenze\\
Viale Morgagni 67/A\\
I-50134 Firenze  \\
{\tt caterina.stoppato@unifi.it}
\end{quote}

\end{document}